\documentclass[11pt]{amsart}

\usepackage[T1]{fontenc}
\usepackage{lmodern}
\usepackage{amsmath,amssymb,amsthm}
\usepackage{booktabs,tabularx}
\usepackage{placeins}
\usepackage{microtype}
\usepackage{tikz}
\usepackage[hidelinks]{hyperref}
\hypersetup{
  pdftitle={The Corners of Ramanujan: Partial- and False-Theta Companions to Third-Order Mock Theta Functions},
  pdfauthor={George E. Andrews and Mohamed El Bachraoui}
}

\newcommand{\qpoch}[3]{(#1;#2)_{#3}}
\newcommand{\qinf}[2]{(#1;#2)_{\infty}}
\newcommand{\co}{c_o}
\newcommand{\cthreeo}{c_{3,o}}
\newcommand{\ptdo}{p_{2,do}}
\newcommand{\pthdo}{p_{3,do}}
\newcommand{\fTheta}{\operatorname{f\Theta}}
\newcommand{\fE}{\operatorname{fE}}
\newcommand{\fO}{\operatorname{fO}}

\theoremstyle{plain}
\newtheorem{theorem}{Theorem}[section]
\newtheorem{proposition}[theorem]{Proposition}
\newtheorem{lemma}[theorem]{Lemma}
\newtheorem{corollary}[theorem]{Corollary}

\theoremstyle{definition}
\newtheorem{definition}[theorem]{Definition}

\theoremstyle{remark}

\numberwithin{equation}{section}

\begin{document}
\raggedbottom

\title[The Corners of Ramanujan]{The Corners of Ramanujan:\\
Partial- and False-Theta Companions to Third-Order Mock Theta Functions}
\author[G. E. Andrews]{George E. Andrews}
\address{The Pennsylvania State University, University Park, Pennsylvania 16802}
\email{andrews@math.psu.edu}

\author[M. El Bachraoui]{Mohamed El Bachraoui}
\address{Department of Mathematical Sciences, United Arab Emirates University,
P.O. Box 15551, Al Ain, United Arab Emirates}
\email{melbachraoui@uaeu.ac.ae}
\begin{abstract}
Most of the traditional third-order mock theta functions occur naturally as
specialized values of partition-rank generating functions.  The same finite
rank kernels also occur in the ordinary and odd unimodal-rank series studied
previously in the literature.  We organize a parallel collection of corner
series by retaining these kernels and replacing their quadratic core weights
by linear peak weights.  Their root-of-unity specializations are governed by
partial and false theta identities recorded in Ramanujan's Lost Notebook; in
one exceptional case the corner companion is itself a mock theta function.
The classical rank series and Ramanujan identities are used as inputs.  Our
new results include exact two- and three-color partition identities with
sparse theta discrepancies and an oriented three-dimensional corner with a
distinct tower.  In both color cases, positive decompositions separate the
colored companion into copies of the relevant odd-corner series and a
corresponding decorated even-corner series.  Thus both sparse discrepancies
become signed differences between natural corner families.  Burson's boxed
modular involution provides a common rank-refined model for the two
corrections, whose surviving fixed points produce the exceptional quadratic
sequences.  For the side rank of the three-dimensional corners we obtain an
exact formula at $z=-1$, a strict parity bias, quartic and cubic root-of-unity
rank-difference identities, and exact equidistribution modulo $3$ on the
progression $6n+5$.  The cubic identity also gives exact congruence formulas
on the progressions $6n+2$ and $6n+3$.  We further classify the associated
binary triangular support modulo every integer and derive general congruence
families.
\end{abstract}
\maketitle

\section{Introduction}

The partition rank reveals a remarkable coherence among several of
Ramanujan's third-order mock theta functions \cite{Watson,AndrewsFour}.
Replacing the underlying partitions by corners---plane partitions having one
row and one column---leads to a closely parallel collection of generating
functions.

We organize the paper according to this parallel, in the spirit of
Ramanujan's own index to the Lost Notebook.  Analytically, the passage from a
mock-theta rank series to the corresponding corner series keeps the finite
rank kernel and replaces the quadratic power of $q$ by the linear power that
records the peak of the corner.  Thus
\[
 n^2\longmapsto n,
 \qquad
 2n(n+1)\longmapsto 2n+1.
\]
We use the descriptive phrase \emph{kernel-preserving linearization} for this
replacement.  It is an organizing principle for the families considered
here, not a claim that arbitrary analytic or modular properties are
preserved.

We have chosen the title \emph{The Corners of Ramanujan} because the analytic
identities needed for the companion formulas are dispersed throughout
Ramanujan's Lost Notebook and acquire a unified corner interpretation here.
The identities themselves are classical; their organization, the colored
interpretations, and the new rank and arithmetic consequences are the focus
of the present work.

Corners have appeared in a variety of mathematical contexts.  The underlying
stack models occur in work of Auluck \cite{Au}, Andrews \cite{A84},
Temperley \cite{T}, and Wright \cite{WrightStacks}; corners also arise as
wedge characters in the work of Datta, Gaberdiel, Li, and Peng \cite{DGLP},
and as V-partitions in Stanley's work \cite{S}.  The word ``corner'' appears
to have been first used by Kreweras \cite{K}; his paper is in French, where
the term is \emph{encoignure}, and the OEIS uses the English translation
``corner'' \cite{O}.

The ordinary corner numbers and their interpretation as unimodal sequences
are classical.  More importantly for the present paper, the two-variable
series $R_c(z;q)$ in \eqref{eq:Rc} is, up to reversal of the rank sign,
exactly the unimodal-rank generating function studied by Kim and Lovejoy
\cite{KLrank}.  They used a partial theta identity of Ramanujan to obtain rank
differences and a congruence for a rank moment, and subsequently developed
analogous identities for several special unimodal families \cite{KLspecial}.
Bridges and Bringmann studied Taylor coefficients of the same rank function
through false Jacobi forms \cite{BridgesBringmann}; Allen and Osburn obtained
two-parameter Hecke--Appell expansions and mixed false-theta representations
for generalized unimodal series \cite{AllenOsburn}; and Bringmann, Pandey,
and van Ittersum developed related partial- and false-Eisenstein-type series
and differential structures \cite{BPV}.  Burson gave a boxed $2$-modular
sign-reversing involution for one of Ramanujan's false theta identities and a
multivariate boxed $m$-modular extension \cite{Burson}.  We show that the
instances $(m,r)=(3,1),(3,2)$ and $(2,1)$ give a common involutive realization
of the rank-refined sparse corrections in the odd and three-dimensional
corner identities.
Root-of-unity specializations are a standard source of rank differences and
cyclotomic equidistribution in partition theory; see, for example,
Garvan \cite{GarvanRank} and Folsom, Males, and Rolen \cite{FMR}.  The
specializations developed here concern the different tower-decorated series
$\mathcal D(z;q)$ introduced in Section~\ref{sec:three-dimensional}.

For the odd family, Andrews's convex compositions provide the earlier
one-variable function \cite{AndrewsCC}, while Bringmann and Lovejoy later
introduced and analyzed the two-variable odd-unimodal rank series used below
\cite[Thm.~1.1]{BLodd}.  Thus neither the definitions of the ordinary and odd
unimodal-rank series nor the previously known rank identities are claimed as
new here.

The paper has three layers.  First, the known rank kernels and classical Lost
Notebook identities are assembled into a single corner/mock-theta dictionary;
this is primarily a synthesis and a change of viewpoint.  Second, we give new
colored-partition interpretations and exact two- and three-color relations
whose errors are sparse theta series.  Parallel positive decompositions
separate the two-color companion into odd corners and staircase-decorated
even corners, and the three-color companion into two copies of the restricted
odd three-dimensional corners and a corresponding decorated even-corner
family.  Both errors are therefore differences of natural corner families.
The two rank-refined correction terms are then placed in one boxed modular
framework: Burson's involutions cancel all nonfixed overpartition pairs,
leaving exactly the exceptional quadratic families.  Third, we develop new rank and arithmetic consequences: an
oriented three-dimensional corner with a distinct tower, an exceptional
parity specialization, quartic and cubic root-of-unity rank-difference
formulas, exact rank equidistribution modulo $3$ on $6n+5$, and an
all-modulus classification of the resulting binary triangular support.  This separation between recalled input, new
interpretation, and new theorem will be maintained throughout the paper.

The arithmetic result has a particularly concise local formulation.  With
$T_j=j(j+1)/2$, define
\[
 R_M=\{T_{2a+1}+T_b\pmod M:a,b\ge0\}.
\]
Theorem~\ref{thm:all-modulus-support} proves that a residue $r$ is absent from
$R_M$ precisely when there is a prime power $p^\alpha\parallel M$ with
$p\equiv3\pmod4$ for which $v_p(4r+1)$ is odd and less than $\alpha$.  This local
description turns the sparse theta discrepancies into congruence families for the
restricted three-dimensional corners and their colored companions.  A
separate cubic specialization of the side-rank series refines one of these
divisibility statements: the three rank classes are exactly equidistributed
at weights $6n+5$, while the rank difference on weights congruent to $2$
modulo $6$ is supported on a single quadratic sequence.  This yields complete
modulo $3$ and modulo $12$ dissections on the corresponding progressions.

The organization is as follows.  Section~\ref{sec:rank-functions} reviews the
two rank kernels, the relevant third-order mock theta functions, and the
partial and false theta functions.  Section~\ref{sec:dictionary} introduces
the parallel corner-rank series and displays the dictionary used in this paper.
Sections \ref{sec:ordinary} and \ref{sec:odd} treat ordinary and odd corners,
respectively; the latter also gives the first positive corner decomposition
and identifies the rank-refined sparse correction with two boxed $3$-modular
involutive series.  Section~\ref{sec:three-dimensional} introduces oriented
three-dimensional corners, consisting of one row, one column, and one tower,
gives the parallel three-color positive decomposition, and identifies the
corresponding correction with the boxed $2$-modular case.
Section~\ref{sec:arithmetic} states the support classification and the
resulting congruence consequences.  Sections~\ref{sec:proofs-ordinary},
\ref{sec:proofs-odd}, and \ref{sec:proofs-three} give, separately and in the
order of their statements, the proofs of the results in Sections
\ref{sec:ordinary}, \ref{sec:odd}, and \ref{sec:three-dimensional}.
Section~\ref{sec:proofs-arithmetic} contains the proofs of the arithmetic
results.  The final section records further directions suggested by the
construction.

\section{Rank functions, mock theta functions, and Ramanujan's identities}
\label{sec:rank-functions}

For $n\ge0$, write
\[
(a;q)_0=1,
\qquad
(a;q)_n=\prod_{j=0}^{n-1}(1-aq^j),
\qquad
(a;q)_\infty=\prod_{j=0}^{\infty}(1-aq^j),
\]
and use the usual condensed notation for several parameters.  For standard
partition and $q$-series notation, see \cite{partitions}.  Throughout,
$|q|<1$ when analytic convergence is relevant; otherwise all identities are
formal power-series identities.  When an intermediate calculation uses
$q^{1/2}$, it is understood in the formal ring $\mathbb Z[[q^{1/2}]]$; the
final identities used below lie in $\mathbb Z[[q]]$.  For a modulus $M$, a
residue is called a square modulo $M$ if it is congruent to $x^2$ for some
integer $x$; otherwise it is called a nonsquare modulo $M$.

Let $N(m,n)$ denote the number of partitions of $n$ with Dyson rank $m$
\cite{Dyson}.  The classical rank generating function is
\begin{equation}\label{eq:R-rank}
R(z;q):=\sum_{n\ge0}\sum_{m\in\mathbb Z}N(m,n)z^mq^n
=\sum_{n\ge0}\frac{q^{n^2}}
{\qpoch{zq}{q}{n}\qpoch{z^{-1}q}{q}{n}}.
\end{equation}
In particular,
\begin{equation}\label{eq:R-specializations}
R(1;q)=\frac1{(q;q)_\infty},
\qquad
R(-1;q)=f_3(q),
\qquad
R(i;q)=\phi_3(q),
\end{equation}
where
\begin{align}
f_3(q)&:=\sum_{n\ge0}\frac{q^{n^2}}{(-q;q)_n^2},
\label{eq:f3}\\
\phi_3(q)&:=\sum_{n\ge0}\frac{q^{n^2}}{(-q^2;q^2)_n}.
\label{eq:phi3}
\end{align}

A second rank-type series is
\begin{equation}\label{eq:S-rank}
S(z;q):=\sum_{n\ge0}
\frac{q^{2n^2+2n}}
{\qpoch{zq}{q^2}{n+1}\qpoch{z^{-1}q}{q^2}{n+1}}.
\end{equation}
It gives
\begin{equation}\label{eq:S-specializations}
S(1;q)=\omega_3(q),
\qquad
S(i;q)=\nu_3(q^2),
\end{equation}
with
\begin{align}
\omega_3(q)&:=\sum_{n\ge0}
\frac{q^{2n(n+1)}}{(q;q^2)_{n+1}^2},
\label{eq:omega3}\\
\nu_3(q)&:=\sum_{n\ge0}
\frac{q^{n(n+1)}}{(-q;q^2)_{n+1}}.
\label{eq:nu3}
\end{align}
We shall also use
\begin{equation}\label{eq:psi3}
\psi_3(q):=\sum_{n\ge1}\frac{q^{n^2}}{(q;q^2)_n}.
\end{equation}

The terminology surrounding unilateral theta-type series is not completely
uniform.  We shall call
\begin{equation}\label{eq:fTheta}
\fTheta(q):=\sum_{n\ge0}(-1)^nq^{n(n+1)/2}
\end{equation}
a partial theta series.  We also define
\begin{align}
\fE(q)&:=\sum_{n\ge0}q^{n(3n+1)/2}(1-q^{2n+1}),
\label{eq:fE}\\
\fO(q)&:=\sum_{n\ge0}q^{3n^2+2n}(1-q^{2n+1}).
\label{eq:fO}
\end{align}
After reindexing the second terms, these have the signed bilateral forms
\begin{align*}
\fE(q)&=\sum_{n\in\mathbb Z}\operatorname{sgn}\!\left(n+\tfrac12\right)
 q^{n(3n+1)/2},\\
\fO(q)&=\sum_{n\in\mathbb Z}\operatorname{sgn}\!\left(n+\tfrac12\right)
 q^{n(3n+2)},
\end{align*}
so we refer to them as false theta series in the standard signed-theta sense.
The subscripts suggest triangular, Euler-pentagonal, and odd-pentagonal
support.  We use ``partial and false theta'' collectively when discussing the
three series.

We shall repeatedly use the $q$-binomial theorem
\begin{equation}\label{eq:qbinomial}
\sum_{n\ge0}\frac{(a;q)_nz^n}{(q;q)_n}
=\frac{(az;q)_\infty}{(z;q)_\infty},
\end{equation}
Euler's pentagonal number theorem
\begin{equation}\label{eq:euler-pentagonal}
(q;q)_\infty=\sum_{k=-\infty}^{\infty}(-1)^kq^{k(3k+1)/2},
\end{equation}
Gauss's identity
\begin{equation}\label{eq:gauss}
G(q):=\sum_{n\ge0}q^{n(n+1)/2}
=\frac{(q^2;q^2)_\infty}{(q;q^2)_\infty},
\end{equation}
Jacobi's triple product
\begin{equation}\label{eq:jacobi-triple}
\sum_{n\in\mathbb Z}(-1)^n x^n q^{n(n-1)/2}
=(x;q)_\infty(q/x;q)_\infty(q;q)_\infty,
\end{equation}
and Heine's transformation
\begin{equation}\label{eq:heine}
\sum_{n\ge0}
\frac{(a;q)_n(b;q)_nz^n}{(q;q)_n(c;q)_n}
=
\frac{(b;q)_\infty(az;q)_\infty}{(c;q)_\infty(z;q)_\infty}
\sum_{n\ge0}
\frac{(c/b;q)_n(z;q)_nb^n}{(q;q)_n(az;q)_n}.
\end{equation}
Heine's transformation was also discovered independently by Ramanujan; see
Berndt \cite[Ch.~16, Entry~6]{B}.

We shall also use the following corollary of the Rogers--Fine identity
\cite[p.~223, Eq.~(9.1.1)]{AB1}:
\begin{equation}\label{eq:rogers-fine}
\sum_{n\ge0}(a;q)_n a^n
=
\sum_{n\ge0}(-1)^nq^{n(3n-1)/2}a^{3n}(1+aq^n).
\end{equation}
The same formula will be applied below with the base $q$ replaced by $q^4$.

Three identities from Part~I of Ramanujan's Lost Notebook are
\cite[Entries~9.3.3, 9.4.2, and 9.5.3]{AB1}
\begin{equation}\label{eq:ram-ch9-1}
\sum_{n\ge0}\frac{(q;q^2)_nq^n}{(-q;q^2)_{n+1}}
=\sum_{n\ge0}(-1)^nq^{2n(n+1)}=\fTheta(q^4),
\end{equation}
\begin{equation}\label{eq:ram-ch9-2}
\sum_{n\ge0}\frac{(-1)^nq^{n(n+1)/2}}{(-q;q)_n}
=\fE(q),
\end{equation}
and
\begin{equation}\label{eq:ram-ch9-3}
\sum_{n\ge0}\frac{q^n}{(-q;q^2)_{n+1}}
=\sum_{n\ge0}(-1)^nq^{6n^2+4n}(1+q^{4n+2})
=\fO(-q^2).
\end{equation}

We also use Burson's boxed modular refinement of
\eqref{eq:ram-ch9-1} \cite[Thm.~5.2]{Burson}.  Recall that an overpartition is
a partition in which the first occurrence of each distinct part may be
overlined \cite{CorteelLovejoy}.  For integers $m\ge2$ and $1\le r<m$, let
$\mathcal B_{m,r}(N)$ be the set of pairs $(k,\overline\pi)$ such that
$k\ge0$, $\overline\pi$ is an overpartition into parts congruent to $r$
modulo $m$ and not exceeding $mk+r$, every overlined part is strictly less
than $mk+r$, and
\[
 rk+|\overline\pi|=N.
\]
Put $\nu(\overline\pi)$ for the number of parts, counted with multiplicity,
and declare $\mathcal B_{m,r}(N)=\varnothing$ when $N<0$.  The corresponding
signed bivariate series is
\begin{align}
\mathfrak B_{m,r}(u,q)
&:=\sum_{N\ge0}\ \sum_{(k,\overline\pi)\in\mathcal B_{m,r}(N)}
 (-1)^{\nu(\overline\pi)}u^{k+\nu(\overline\pi)}q^N\notag\\
&=\sum_{k\ge0}
 \frac{(uq^r;q^m)_k u^kq^{rk}}{(-uq^r;q^m)_{k+1}}\notag\\
&=\sum_{j\ge0}(-1)^ju^{2j}q^{j(mj+2r)}.
\label{eq:burson-boxed}
\end{align}
For fixed $k$, the product has the stated overpartition interpretation: a
nonmaximal permitted part $s$ contributes
$(1-uq^s)/(1+uq^s)$, the maximal part $mk+r$ contributes
$1/(1+uq^{mk+r})$, and the boxed row contributes $u^kq^{rk}$.
Recall that an involution $\iota$ on a signed set satisfies
$\iota^2=\mathrm{id}$.  An object $X$ is called a \emph{fixed point} when
$\iota(X)=X$.  Every nonfixed object belongs to a two-cycle
$\{X,\iota(X)\}$; if $\iota$ reverses the sign, the two members of that
cycle cancel in the signed generating function.  Thus only the fixed
points contribute to the final signed sum.

Burson's boxed $m$-modular involution preserves both $N$ and
$k+\nu(\overline\pi)$ and reverses the sign on every nonfixed pair.  Its
fixed points are the pairs
\begin{equation}\label{eq:burson-fixed-pairs}
 \bigl(j,(mj+r)^j\bigr),
\end{equation}
where $(s)^j$ denotes the partition consisting of $j$ copies of the part
$s$.  These fixed points have weight $(-1)^ju^{2j}q^{j(mj+2r)}$.

The next two identities are from Part~II \cite[Ch.~6]{AB2}.  First,
\begin{align}
\sum_{n\ge0}\frac{q^n}{(-aq;q)_n(-bq;q)_n}
&=\frac{1}{(-aq;q)_\infty(-bq;q)_\infty}
\sum_{m\ge0}(-1)^{m+1}a^{-m-1}b^mq^{\binom{m+1}{2}}
\notag\\
&\quad+\left(1+\frac1a\right)
\sum_{m\ge0}\frac{(-1)^ma^{-m}b^mq^{\binom{m+1}{2}}}
{(-bq;q)_m}.
\label{eq:ram-part2-1}
\end{align}
Second,
\begin{align}
&\left(1+\frac1a\right)
\sum_{n\ge0}\frac{(-q;q)_{2n}q^{2n+1}}
{(aq;q^2)_{n+1}(q/a;q^2)_{n+1}}
\notag\\
&\quad=\frac{1}{(q;q^2)_\infty(aq;q^2)_\infty(q/a;q^2)_\infty}
\sum_{n\ge0}(-a)^nq^{\binom{n+1}{2}}
-\sum_{n\ge0}(-a)^nq^{n^2+n}.
\label{eq:ram-part2-3}
\end{align}

We shall also need the following companion identity from Part~III of
Ramanujan's Lost Notebook \cite[p.~60, Eq.~(8.25)]{AB3}:
\begin{equation}\label{eq:ram-part3-psi}
\sum_{n\ge1}\frac{q^n}{(q;q^2)_n}
=q\sum_{n\ge0}\frac{q^{2n(n+1)}}{(q;q^2)_{n+1}^2}.
\end{equation}

\section{Corner ranks and the mock-theta dictionary}
\label{sec:dictionary}

A corner is a plane partition whose Ferrers diagram has one row and one
column.  Equivalently, it consists of a central part $m$, a row partition,
and a column partition, with every row and column part at most $m$.  If the
row has $r$ parts and the column has $s$ parts, the corner rank is $r-s$.
Let $C(r,n)$ denote the number of corners of weight $n$ and corner rank $r$.
Thus
\begin{equation}\label{eq:Rc}
R_c(z;q):=\sum_{n\ge0}\sum_{r\in\mathbb Z}C(r,n)z^rq^n
=\sum_{m\ge0}\frac{q^m}
{(zq;q)_m(z^{-1}q;q)_m}.
\end{equation}
Up to reversing the sign convention for the rank, \eqref{eq:Rc} is exactly
the unimodal-rank generating function $U(z;q)$ of Kim and Lovejoy
\cite[Eq.~(2.2)]{KLrank}.  The reversal is immaterial because the series is
invariant under $z\mapsto z^{-1}$.

If every part is required to be odd, the corresponding series is
\begin{equation}\label{eq:Sc}
S_c(z;q):=\sum_{m\ge0}\frac{q^{2m+1}}
{(zq;q^2)_{m+1}(z^{-1}q;q^2)_{m+1}}.
\end{equation}
This is the odd-unimodal rank series of Bringmann and Lovejoy
\cite[Thm.~1.1]{BLodd}.  The denominators in \eqref{eq:Rc} and
\eqref{eq:Sc} are exactly those of \eqref{eq:R-rank} and \eqref{eq:S-rank};
only the quadratic core weights have been linearized.

The resulting dictionary is summarized in Table~\ref{tab:dictionary}.
The entries in the last column are developed in the following sections.

\begin{table}[t]
\centering
\caption{Third-order rank specializations and their corner companions.}
\label{tab:dictionary}
\renewcommand{\arraystretch}{1.25}
\begin{tabularx}{\textwidth}{@{}lclX@{}}
\toprule
Rank kernel & Specialization & Classical series & Corner companion \\
\midrule
$R(z;q)$ & $z=1$ & $(q;q)_\infty^{-1}$
& $R_c(1;q)$, ordinary corners \\
$R(z;q)$ & $z=-1$ & $f_3(q)$
& $R_c(-1;q)$, governed by $\fTheta$ and $\fE$ \\
$R(z;q)$ & $z=i$ & $\phi_3(q)$
& $R_c(i;q)=\fE(-q)$ \\
$S(z;q)$ & $z=1$ & $\omega_3(q)$
& $S_c(1;q)$, odd corners and a sparse $\fO(-q)$ term \\
$S(z;q)$ & $z=i$ & $\nu_3(q^2)$
& $S_c(i;q)=q\fO(-q^4)$ after rescaling \\
exceptional & --- & $\psi_3(q)$
& $C_\psi(q)=q\omega_3(q)$, itself mock \\
\bottomrule
\end{tabularx}
\end{table}

The table is a guide rather than a claim that all analytic properties of a
mock theta function survive linearization.  Its force is that one fixed
rank kernel produces both a mock-theta specialization and a corner series,
while Ramanujan's identities evaluate the latter through partial and false
theta functions.

The present dictionary does not include the remaining classical third-order
functions $\chi_3$ and $\rho_3$.  No claim of a complete dictionary for every
third-order mock theta function is therefore intended.

\section{Ordinary corner companions}
\label{sec:ordinary}

We state the results first and give their proofs together in
Section~\ref{sec:proofs-ordinary}.

Let $c(n)$ denote the number of corners of $n$.  Then
\begin{equation}\label{eq:c-gf}
C(q):=\sum_{n\ge0}c(n)q^n
=R_c(1;q)=\sum_{m\ge0}\frac{q^m}{(q;q)_m^2}.
\end{equation}
For example, $c(4)=12$; the twelve objects are
\[
4,\ 31,\ \begin{smallmatrix}3\\1\end{smallmatrix},\ 22,\
\begin{smallmatrix}2\\2\end{smallmatrix},\ 211,\
\begin{smallmatrix}2&1\\1\end{smallmatrix},\
\begin{smallmatrix}2\\1\\1\end{smallmatrix},\ 1111,\
\begin{smallmatrix}1&1&1\\1\end{smallmatrix},\
\begin{smallmatrix}1&1\\1\\1\end{smallmatrix},\
\begin{smallmatrix}1\\1\\1\\1\end{smallmatrix}.
\]
The same objects are ordinary unimodal sequences with distinguished peak.
The following is Ramanujan's partial-theta identity in the form used by
Kim and Lovejoy.

\begin{proposition}
\label{prop:C-fTheta}
We have
\begin{equation}\label{eq:C-fTheta}
C(q)=\frac{\fTheta(q)}{(q;q)_\infty^2}.
\end{equation}
\end{proposition}

At the level of unimodal sequences, this is the specialization of the
Ramanujan partial-theta identity used by Kim and Lovejoy \cite{KLrank}.  We
record it in corner language because it is the starting point for the new
colored interpretation below.  We are not aware of a comparably direct
bijective proof of \eqref{eq:C-fTheta}.

\begin{theorem}\label{thm:ordinary-colored}
For $n\ge0$, the following quantities are equal:
\begin{enumerate}
\item $c(n)$;
\item the number of two-color partitions of $n+1$ in which the largest part
occurs exactly once, and that occurrence is red;
\item the number of two-color partitions of $n$, with parts colored red and
blue, in which the smallest positive integer absent as a blue part is odd.
\end{enumerate}
\end{theorem}

\begin{proposition}\label{prop:c-parity}
The number $c(n)$ is odd if and only if $n$ is a generalized pentagonal
number, that is,
\[
 n=\frac{k(3k\pm1)}2
\]
for some $k\ge0$.
\end{proposition}

\begin{corollary}\label{cor:c-residue}
Let $M\ge2$ and $0\le r<M$.  If $24r+1$ is a nonsquare modulo $M$, then
\[
 c(Mn+r)\equiv0\pmod2\qquad(n\ge0).
\]
\end{corollary}

The specialization $z=-1$ next supplies the corner companion of $f_3$:
\begin{equation}\label{eq:Cf-def}
C_f(q):=R_c(-1;q)=\sum_{m\ge0}\frac{q^m}{(-q;q)_m^2}.
\end{equation}
Its coefficient of $q^n$ is the number of corners of $n$ with even rank
minus the number with odd rank.

\begin{proposition}\label{prop:Cf}
We have
\begin{equation}\label{eq:Cf}
C_f(q)=-\frac{\fTheta(q)}{(-q;q)_\infty^2}+2\fE(q).
\end{equation}
\end{proposition}

Finally, for $z=i$ we obtain the companion of $\phi_3$:
\begin{equation}\label{eq:Cphi-def}
C_\phi(q):=R_c(i;q)
=\sum_{m\ge0}\frac{q^m}{(-q^2;q^2)_m}.
\end{equation}

\begin{proposition}\label{prop:Cphi}
We have
\begin{align}
C_\phi(q)
&=\sum_{m\ge0}(q;q^4)_m q^m
+q^2\sum_{m\ge0}(q^3;q^4)_m q^{3m}
\label{eq:Cphi-finite}\\
&=\fE(-q).
\label{eq:Cphi-fE}
\end{align}
\end{proposition}

The two preceding propositions are direct evaluations of the indicated
root-of-unity specializations by the classical identities in
Section~\ref{sec:rank-functions}.  Their role is to complete the dictionary;
no novelty is claimed for the underlying $q$-series transformations.

\section{Odd corner companions}
\label{sec:odd}

We state the results first and give their proofs together in
Section~\ref{sec:proofs-odd}.

Let $\co(n)$ denote the number of nonempty corners of $n$ for which the peak,
row parts, and column parts are all odd.  Then
\begin{equation}\label{eq:Co-gf}
C_\omega(q):=\sum_{n\ge0}\co(n)q^n
=S_c(1;q)=\sum_{m\ge0}\frac{q^{2m+1}}{(q;q^2)_{m+1}^2}.
\end{equation}
For example, $\co(5)=9$:
\[
5,\ 311,\ \begin{smallmatrix}3&1\\1\end{smallmatrix},\
\begin{smallmatrix}3\\1\\1\end{smallmatrix},\ 11111,\
\begin{smallmatrix}1&1&1&1\\1\end{smallmatrix},\
\begin{smallmatrix}1&1&1\\1\\1\end{smallmatrix},\
\begin{smallmatrix}1&1\\1\\1\\1\end{smallmatrix},\
\begin{smallmatrix}1\\1\\1\\1\\1\end{smallmatrix}.
\]
After increasing the peak by $2$, this is Andrews's one-variable odd
convex-composition function \cite{AndrewsCC}.  Bringmann and Lovejoy later
developed the two-variable odd-unimodal rank series
\begin{equation}\label{eq:Oz-gf}
\mathcal O(z;q)=S_c(z;q)
=\sum_{m\ge0}\frac{q^{2m+1}}
{(zq;q^2)_{m+1}(z^{-1}q;q^2)_{m+1}}.
\end{equation}
Their Theorem~1.1 gives
\begin{align}
\mathcal O(z;q)
&=\sum_{j\ge0}(-1)^{j+1}z^{3j+1}q^{3j^2+2j}(1+zq^{2j+1})
\notag\\
&\quad+\frac{1}{(zq;q^2)_\infty(z^{-1}q;q^2)_\infty}
\sum_{j\ge0}(-1)^jz^{2j+1}q^{j(j+1)}.
\label{eq:Oz-ramanujan}
\end{align}
The first line of \eqref{eq:Oz-ramanujan} has a common boxed modular
interpretation.  The square reparametrization $z=w^2$ avoids fractional
powers of the rank variable.  The precise boxed $3$-modular identity is as
follows.

\begin{proposition}
\label{prop:odd-boxed-correction}
For a formal variable $w$,
\begin{align}
&\sum_{j\ge0}(-1)^{j+1}w^{6j+2}q^{3j^2+2j}
 \left(1+w^2q^{2j+1}\right)\notag\\
&\qquad=-w^2\mathfrak B_{3,1}(w^3,q)
          -w^4q\mathfrak B_{3,2}(w^3,q).
\label{eq:odd-boxed-correction}
\end{align}
Thus the complete rank-refined sparse term in
\eqref{eq:Oz-ramanujan}, after $z=w^2$, is governed by two instances of the
single boxed modular involution \eqref{eq:burson-boxed}.
\end{proposition}

From \eqref{eq:Co-gf} and \eqref{eq:Oz-gf}--\eqref{eq:Oz-ramanujan},
applied with $z=1$, we obtain
\begin{equation}\label{eq:Co-second}
C_\omega(q)
=-\sum_{j\ge0}(-1)^jq^{3j^2+2j}(1+q^{2j+1})
+\frac{1}{(q;q^2)_\infty^2}
\sum_{j\ge0}(-1)^jq^{j(j+1)}.
\end{equation}
Thus the same denominator that defines $\omega_3(q)$ now carries the linear
peak weight $2m+1$ rather than the quadratic weight $2m(m+1)$.

The simple largest-part interpretation survives in the odd setting.

\begin{theorem}\label{thm:Co-colored}
For $n\ge0$, $\co(n)$ equals the number of two-color partitions of $n+2$
into odd parts in which the largest part occurs exactly once, and that
occurrence is red.
\end{theorem}

We next define the smallest-part companion needed for the exact discrepancy.
Let $\ptdo(n)$ count two-color partitions of $n$ with smallest part size $s$
such that at least one smallest part is red, the red part of size $s$ occurs
once or twice and all larger red parts are distinct, while the blue parts are
odd and at least $2s-1$.  Put $\ptdo(0)=0$.  Then
\begin{equation}\label{eq:Ptdo-gf}
\sum_{n\ge0}\ptdo(n)q^n
=\sum_{s\ge1}q^s\frac{(-q^s;q)_\infty}{(q^{2s-1};q^2)_\infty}.
\end{equation}
The definition is an analytic companion: fixing the smallest part $s$
produces exactly the product needed in the transformation leading to
Theorem~\ref{thm:Co-Ptdo}.

For $a\ge0$, let
\[
 \delta_a=(a,a-1,\ldots,2,1),
 \qquad |\delta_a|=T_a=\frac{a(a+1)}2,
\]
with $\delta_0$ empty.  If $\Gamma$ is an ordinary corner, write $2\Gamma$
for the even corner obtained by doubling every entry.  A
\emph{staircase-decorated even corner} is a triple
\[
 (2\Gamma,\delta_a,\delta_b)
\]
of weight $2|\Gamma|+T_a+T_b$, with the two staircases regarded as
independent decorations of the labeled row and column arms.  Let $e_\triangle(N)$ denote the number of
such triples of weight $N$.  By \eqref{eq:c-gf} and \eqref{eq:gauss},
\begin{equation}\label{eq:Etriangle-gf}
\begin{aligned}
 E_\triangle(q)
 &:=\sum_{N\ge0}e_\triangle(N)q^N\\
 &=\left(\sum_{a\ge0}q^{T_a}\right)^2
   \left(\sum_{k\ge0}\frac{q^{2k}}{(q^2;q^2)_k^2}\right)\\
 &=G(q)^2C(q^2).
\end{aligned}
\end{equation}

The two-color companion admits the following positive decomposition into
two corner families.

\begin{theorem}\label{thm:Ptdo-positive-decomposition}
We have
\begin{equation}\label{eq:Ptdo-positive-decomposition}
 \sum_{N\ge0}\ptdo(N+1)q^N
 =C_\omega(q)+E_\triangle(q).
\end{equation}
Equivalently, for every $N\ge0$,
\begin{equation}\label{eq:Ptdo-positive-coeff}
 \ptdo(N+1)=\co(N)+e_\triangle(N).
\end{equation}
\end{theorem}

Thus the comparison between the two-color companion and two copies of the
odd corners can be rewritten as a comparison between two corner families.
A direct sign-reversing involution on staircase-decorated even corners and
odd corners would prove the sparse discrepancy combinatorially.  The boxed
modular model below supplies the required fixed points, but not yet such a
direct corner-to-corner involution.

\begin{theorem}\label{thm:Co-Ptdo}
For $n\ge0$,
\begin{equation}\label{eq:Ptdo-theta-coeff}
\ptdo(n+1)-2\co(n)=
\begin{cases}
(-1)^j,&n=3j^2+2j,\\
(-1)^j,&n=3j^2+4j+1,\\
0,&\text{otherwise.}
\end{cases}
\end{equation}
\end{theorem}

Combining the positive decomposition with Burson's boxed modular identity
gives the following corner-to-corner and signed-set descriptions.

\begin{corollary}
\label{cor:two-color-boxed}
For every $N\ge0$,
\begin{align}
 e_\triangle(N)-\co(N)
 &=\ptdo(N+1)-2\co(N)\notag\\
 &=\sum_{(k,\overline\pi)\in\mathcal B_{3,1}(N)}
   (-1)^{\nu(\overline\pi)}\notag\\
 &\quad+\sum_{(k,\overline\pi)\in\mathcal B_{3,2}(N-1)}
   (-1)^{\nu(\overline\pi)}.
\label{eq:two-color-boxed}
\end{align}
The second sum is understood to be empty when $N=0$.  Under Burson's
involutions, the first signed sum is nonzero precisely when
$N=3j^2+2j$, and the second is nonzero precisely when
$N=3j^2+4j+1$; in either case the surviving value is $(-1)^j$.
\end{corollary}

\begin{proposition}\label{prop:Co-parity}
For $n\ge0$, $\co(n)$ is odd if and only if $3n+1$ is an even square.
\end{proposition}

\begin{corollary}\label{cor:Ptdo-mod4}
Let $\chi_3$ be the nontrivial character modulo $3$.  For $n\ge0$,
\[
\ptdo(n)\equiv
\begin{cases}
\chi_3(x),&3n-2=x^2\text{ for some }x>0,\\
0,&3n-2\text{ is not a square}
\end{cases}
\pmod4.
\]
\end{corollary}

\begin{corollary}\label{cor:two-nonsquare}
Let $M\ge2$ and $0\le r<M$.  If $3r+1$ is a nonsquare modulo $M$, then
$\co(Mn+r)\equiv0\pmod2$.  If $3r-2$ is a nonsquare modulo $M$, then
$\ptdo(Mn+r)\equiv0\pmod4$.
\end{corollary}

The $z=i$ specialization gives the companion of $\nu_3$.  Put
\begin{equation}\label{eq:Cnu-def}
C_\nu(q):=\sum_{m\ge0}\frac{q^m}{(-q;q^2)_{m+1}}.
\end{equation}
Consequently, $qC_\nu(q^2)=S_c(i;q)$ is the signed generating function
for self-conjugate odd corners: a side partition is repeated on the two arms
and contributes the sign $(-1)^{\ell(\lambda)}$.

\begin{proposition}\label{prop:Cnu}
We have
\begin{equation}\label{eq:Cnu}
C_\nu(q)=\sum_{m\ge0}(q^2;q^4)_m q^{2m}=\fO(-q^2).
\end{equation}
\end{proposition}

Finally, define
\begin{equation}\label{eq:Cpsi-def}
C_\psi(q):=\sum_{m\ge1}\frac{q^m}{(q;q^2)_m}.
\end{equation}
This series is included because it has the same one-arm odd denominator as
the $\psi_3$ kernel after linearizing the quadratic weight.  It is the
exceptional member of the dictionary.

\begin{proposition}\label{prop:Cpsi}
We have
\begin{equation}\label{eq:Cpsi}
C_\psi(q)=q\omega_3(q).
\end{equation}
\end{proposition}

Thus the expected partial- or false-theta companion of $\psi_3$ remains in
the mock-theta world.  The two preceding propositions are recorded to
complete the dictionary; their analytic content follows from the classical
identities in Section~\ref{sec:rank-functions}.  This exception is one reason
to regard the dictionary as a structured parallel rather than a formal
operator on arbitrary $q$-series.

\section{Oriented three-dimensional corners and a three-color companion}
\label{sec:three-dimensional}

We state the results first and give their proofs together in
Section~\ref{sec:proofs-three}.

The obvious extension of a corner to three dimensions has one row, one
column, and one tower.  To make this precise, and to avoid confusion with
solid partitions in the usual order-ideal sense, we retain labels on the
three coordinate arms.

\begin{definition}\label{def:three-arm}
An \emph{oriented three-dimensional corner} is a quadruple
\[
\Gamma=(m;\lambda,\mu,\nu),
\]
where $m\ge0$ and $\lambda,\mu,\nu$ are ordinary partitions whose parts are
at most $m$.  Its weight is
\[
|\Gamma|=m+|\lambda|+|\mu|+|\nu|.
\]
The three partitions occupy labeled coordinate arms; no quotient by a
permutation or rotation is taken.
\end{definition}

Only the following restricted family is used in this paper.  Its row and
column retain the odd-corner structure, while its tower is a distinct-part
decoration.

\begin{definition}\label{def:decorated-three-arm}
A \emph{restricted odd three-dimensional corner} is a tuple
\[
(2m+1;\lambda,\mu,\nu),\qquad m\ge0,
\]
in which every part of $\lambda$ and $\mu$ is odd and at most $2m+1$, while
$\nu$ has distinct parts in $\{1,2,\ldots,2m\}$.  We write $\cthreeo(n)$ for
the number of such objects of weight $n$ and set $\cthreeo(0)=0$.
\end{definition}

The restriction is asymmetric: the row and column form an odd-corner
skeleton, whereas the tower is a bounded distinct-part decoration.  The
orientation is part of the structure.  We select this family because its
finite tower product is precisely the factor occurring in Ramanujan's
identity \eqref{eq:ram-part2-3}; no rotationally invariant solid-partition
interpretation is intended.

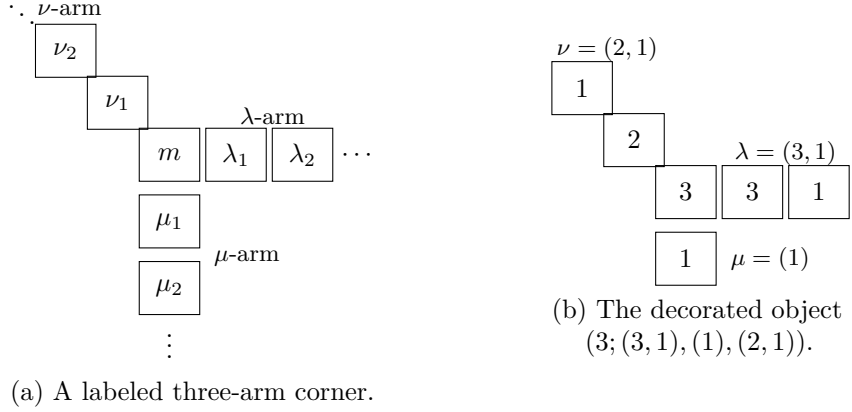
\begin{figure}[t]
\centering
\begin{minipage}{0.47\textwidth}
\centering
\begin{tikzpicture}[scale=0.88, every node/.style={font=\small}]
\tikzset{armbox/.style={draw,minimum width=8mm,minimum height=7mm,inner sep=1pt}}
\node[armbox] (m) at (0,0) {$m$};
\node[armbox] at (1.0,0) {$\lambda_1$};
\node[armbox] at (2.0,0) {$\lambda_2$};
\node at (2.85,0) {$\cdots$};
\node[armbox] at (0,-1.0) {$\mu_1$};
\node[armbox] at (0,-2.0) {$\mu_2$};
\node at (0,-2.75) {$\vdots$};
\node[armbox] at (-0.78,0.78) {$\nu_1$};
\node[armbox] at (-1.56,1.56) {$\nu_2$};
\node at (-2.22,2.22) {$\ddots$};
\node[font=\footnotesize] at (1.55,0.58) {$\lambda\text{-arm}$};
\node[font=\footnotesize,anchor=west] at (0.52,-1.55) {$\mu\text{-arm}$};
\node[font=\footnotesize] at (-1.50,2.18) {$\nu\text{-arm}$};
\end{tikzpicture}

\small (a) A labeled three-arm corner.
\end{minipage}
\hfill
\begin{minipage}{0.47\textwidth}
\centering
\begin{tikzpicture}[scale=0.88, every node/.style={font=\small}]
\tikzset{armbox/.style={draw,minimum width=8mm,minimum height=7mm,inner sep=1pt}}
\node[armbox] at (0,0) {$3$};
\node[armbox] at (1.0,0) {$3$};
\node[armbox] at (2.0,0) {$1$};
\node[armbox] at (0,-1.0) {$1$};
\node[armbox] at (-0.78,0.78) {$2$};
\node[armbox] at (-1.56,1.56) {$1$};
\node[font=\footnotesize] at (1.50,0.58) {$\lambda=(3,1)$};
\node[font=\footnotesize,anchor=west] at (0.52,-1.0) {$\mu=(1)$};
\node[font=\footnotesize] at (-1.18,2.15) {$\nu=(2,1)$};
\end{tikzpicture}

\small (b) The decorated object $(3;(3,1),(1),(2,1))$.
\end{minipage}
\caption{Schematic labeled-arm diagrams.  In (b), the first two arms are odd,
the decorated arm is distinct and below the peak, the weight is $11$, and the
side rank is $2-1=1$.}
\label{fig:three-arm}
\end{figure}

For peak $2m+1$, the two odd arms contribute
$1/(q;q^2)_{m+1}^2$, and the distinct tower contributes $(-q;q)_{2m}$.
Hence
\begin{equation}\label{eq:Cto-gf}
C_{3,o}(q):=\sum_{n\ge0}\cthreeo(n)q^n
=\sum_{m\ge0}\frac{(-q;q)_{2m}q^{2m+1}}{(q;q^2)_{m+1}^2}.
\end{equation}
For a restricted odd three-dimensional corner
$\Gamma=(2m+1;\lambda,\mu,\nu)$, define its side rank by
\[
 \operatorname{rank}(\Gamma)=\ell(\lambda)-\ell(\mu).
\]
Let $d(r,n)$ denote the number of such corners of weight $n$ and side rank
$r$.  A rank variable on the two side arms then gives
\begin{equation}\label{eq:Dz-gf}
\mathcal D(z;q):=\sum_{m\ge0}
\frac{(-q;q)_{2m}q^{2m+1}}
{(zq;q^2)_{m+1}(z^{-1}q;q^2)_{m+1}}
=\sum_{n\ge0}\sum_{r\in\mathbb Z}d(r,n)z^rq^n.
\end{equation}
For $\ell\ge2$ and $0\le a<\ell$, let
\begin{equation}\label{eq:rank-residue-counts}
 d_a^{(\ell)}(n):=\sum_{r\equiv a\pmod\ell}d(r,n).
\end{equation}
Interchanging the row and column fixes the class $a=0$ and gives
\[
 d_a^{(\ell)}(n)=d_{\ell-a}^{(\ell)}(n)
 \qquad(1\le a<\ell).
\]
Putting $a=z$ in \eqref{eq:ram-part2-3} yields
\begin{align}
(1+z^{-1})\mathcal D(z;q)
&=\frac{1}{(q;q^2)_\infty(zq;q^2)_\infty(z^{-1}q;q^2)_\infty}
\sum_{j\ge0}(-z)^jq^{\binom{j+1}{2}}
\notag\\
&\quad-\sum_{j\ge0}(-z)^jq^{j(j+1)}.
\label{eq:Dz-ramanujan}
\end{align}
The second term in \eqref{eq:Dz-ramanujan} is the boxed $2$-modular member
of the same family.  Its precise specialization is the following.

\begin{proposition}
\label{prop:three-boxed-correction}
For a formal variable $w$,
\begin{equation}\label{eq:three-boxed-correction}
 \sum_{j\ge0}(-w^2)^jq^{j(j+1)}
 =\mathfrak B_{2,1}\left(w,q^{1/2}\right).
\end{equation}
Consequently, after $z=w^2$, the rank-refined sparse correction in
\eqref{eq:Dz-ramanujan} is precisely the fixed-point contribution of
Burson's boxed $2$-modular involution.
\end{proposition}

The point $z=-1$ is exceptional in \eqref{eq:Dz-ramanujan}, because the
factor $1+z^{-1}$ vanishes.  It nevertheless yields an exact rank identity.
Retain the triangular-number series $G(q)$ from \eqref{eq:gauss}, and put
\begin{equation}\label{eq:G-G1}
 G_1(q):=\sum_{j\ge0}j q^{j(j+1)/2}.
\end{equation}

\begin{theorem}\label{thm:D-parity-specialization}
We have
\begin{align}
 \mathcal D(-1;q)
 &=\frac{G(q^2)}{G(q)}G_1(q)-G_1(q^2)\label{eq:D-parity-specialization}\\
 &=\frac{(q;q^2)_\infty}{(q^2;q^4)_\infty^2}G_1(q)-G_1(q^2).
 \label{eq:D-parity-specialization-product}
\end{align}
\end{theorem}

By \eqref{eq:rank-residue-counts},
\[
 \mathcal D(-1;q)=\sum_{n\ge0}
 \bigl(d_0^{(2)}(n)-d_1^{(2)}(n)\bigr)q^n.
\]
The next theorem gives a strict bias between the two parity classes.

\begin{theorem}\label{thm:D-parity-bias}
We have
\begin{equation}\label{eq:D-parity-positive}
 -\mathcal D(-1;-q)
 =\sum_{m\ge0}
 \frac{q^{2m+1}(-q^2;q^2)_m}
 {(q;q^2)_m(1-q^{2m+1})^2}.
\end{equation}
Consequently, for every $n\ge1$,
\begin{equation}\label{eq:D-parity-inequality}
 (-1)^{n+1}\bigl(d_0^{(2)}(n)-d_1^{(2)}(n)\bigr)>0.
\end{equation}
Thus even side ranks predominate when $n$ is odd, whereas odd side ranks
predominate when $n$ is even.
\end{theorem}

The right-hand side of \eqref{eq:D-parity-positive} has a direct positive
partition interpretation.  For fixed $m$, it records a distinguished part
$2m+1$, arbitrary odd parts below the peak, distinct even parts not exceeding
$2m$, and two independent colors for additional copies of the peak.

The nonexceptional fourth and third roots of unity give two further exact
rank differences.  Write $T_j=j(j+1)/2$ and define
\begin{align}
 \Phi_4(q)&:=\sum_{m\ge0}(-1)^m q^{T_{2m}}
 =\sum_{m\ge0}(-1)^m q^{m(2m+1)},\label{eq:Phi4}\\
 \Phi_3(q)&:=\sum_{m\ge0}(-1)^m
 \left(q^{T_{3m}}+q^{T_{3m+1}}\right).\label{eq:Phi3}
\end{align}

\begin{theorem}
\label{thm:D-cyclotomic-differences}
Let $\omega=e^{2\pi i/3}$.  Then
\begin{align}
 \mathcal D(i;q)
 &=\sum_{n\ge0}\left(d_0^{(4)}(n)-d_2^{(4)}(n)\right)q^n\notag\\
 &=\frac{\Phi_4(q)}{(q;q^2)_\infty(-q^2;q^4)_\infty}
   -\Phi_4(q^2),\label{eq:D-i-difference}
\end{align}
and
\begin{align}
 \mathcal D(\omega;q)
 &=\sum_{n\ge0}\left(d_0^{(3)}(n)-d_1^{(3)}(n)\right)q^n\notag\\
 &=\frac{\Phi_3(q)}{(q^3;q^6)_\infty}
   -\Phi_3(q^2).\label{eq:D-omega-difference}
\end{align}
\end{theorem}

The cubic identity first gives exact rank equidistribution on the
progression $6n+5$.

\begin{corollary}
\label{cor:cubic-rank-equidistribution}
For every $n\ge0$,
\begin{equation}\label{eq:cubic-equidistribution}
 d_0^{(3)}(6n+5)=d_1^{(3)}(6n+5)=d_2^{(3)}(6n+5)
 =\frac{\cthreeo(6n+5)}{3}.
\end{equation}
Equivalently, the side-rank Laurent polynomial
\[
 \sum_{r\in\mathbb Z}d(r,6n+5)z^r
\]
is divisible by $1+z+z^2$ in $\mathbb Z[z,z^{-1}]$.
\end{corollary}

The same identity gives a sparse rank difference on the progression
$6n+2$.

\begin{corollary}
\label{cor:cubic-sparse-rank-difference}
If $N\equiv2\pmod6$, then
\begin{equation}\label{eq:cubic-sparse-rank-difference}
 d_0^{(3)}(N)-d_1^{(3)}(N)
 =\begin{cases}
 (-1)^{m+1},&\substack{N=(3m+1)(3m+2)\\ \text{for some }m\ge0},\\
 0,&\text{otherwise.}
 \end{cases}
\end{equation}
\end{corollary}

Let $\pthdo(n)$ count three-color partitions of $n$, with colors red, blue,
and yellow, and smallest part size $s$, subject to the following conditions:
at least one smallest part is red; the red size $s$ occurs once or twice and
all larger red parts are distinct; the blue parts are distinct and at least
$s$; and the yellow parts are odd and at least $2s-1$.  Set
$\pthdo(0)=0$.  Then
\begin{equation}\label{eq:Pthdo-gf}
\sum_{n\ge0}\pthdo(n)q^n
=\sum_{s\ge1}q^s\frac{(-q^s;q)_\infty^2}{(q^{2s-1};q^2)_\infty}.
\end{equation}
As in the two-color case, this is an analytic companion selected so that a
fixed smallest part produces the product in the transformed corner series.
We next give a parallel positive decomposition.

For $k\ge0$, let $\Gamma=(k;\alpha,\beta)$ be an ordinary corner and let
$2\Gamma=(2k;2\alpha,2\beta)$ be obtained by doubling every entry.  For
$k\ge1$, a \emph{zero-augmented distinct tower below $2k$} is a subset
\[
 \eta\subseteq\{0,1,\ldots,2k-1\}.
\]
For $k=0$, the allowable tower set is understood to be empty and we take
$\eta=\varnothing$, in agreement with $(-1;q)_0=1$.  In all cases, put
$\operatorname{wt}(\eta)=\sum_{t\in\eta}t$; the possible entry $0$ is a
mark, not a positive part.  A \emph{staircase-and-tower-decorated
even corner} is a quadruple
\[
 (2\Gamma,\delta_a,\delta_b,\eta)
\]
of weight
\[
 2|\Gamma|+T_a+T_b+\operatorname{wt}(\eta).
\]
Let $e_{3,\triangle}(N)$ count these objects of weight $N$.  Since the
zero-augmented tower contributes $(-1;q)_{2k}$, their generating function is
\begin{equation}\label{eq:E3triangle-gf}
 E_{3,\triangle}(q):=\sum_{N\ge0}e_{3,\triangle}(N)q^N
 =G(q)^2\sum_{k\ge0}
 \frac{(-1;q)_{2k}q^{2k}}{(q^2;q^2)_k^2}.
\end{equation}

The three-color companion has a positive decomposition parallel to the
two-color decomposition above.

\begin{theorem}
\label{thm:Pthdo-positive-decomposition}
We have
\begin{equation}\label{eq:Pthdo-positive-decomposition}
 \sum_{N\ge0}\pthdo(N+1)q^N
 =2C_{3,o}(q)+E_{3,\triangle}(q).
\end{equation}
Equivalently, for every $N\ge0$,
\begin{equation}\label{eq:Pthdo-positive-coeff}
 \pthdo(N+1)=2\cthreeo(N)+e_{3,\triangle}(N).
\end{equation}
\end{theorem}

Thus the three-color companion splits positively into two labeled copies of
the restricted odd three-dimensional corners and one decorated even-corner
family.  A direct bijection realizing this decomposition, and a direct
sign-reversing involution between the latter family and two labeled copies of
the restricted corners, remain to be found.  Proposition~\ref{prop:three-boxed-correction}
provides the fixed-point model for the sparse correction.

\begin{theorem}\label{thm:Cto-Pthdo}
For $n\ge0$,
\begin{equation}\label{eq:Pthdo-theta-coeff}
\pthdo(n+1)-4\cthreeo(n)=
\begin{cases}
(-1)^j,&n=j(j+1),\\
0,&\text{otherwise.}
\end{cases}
\end{equation}
\end{theorem}

Combining the preceding decomposition with the boxed $2$-modular identity
gives the following two forms of the three-color discrepancy.

\begin{corollary}
\label{cor:three-color-boxed}
For every $N\ge0$,
\begin{align}
 e_{3,\triangle}(N)-2\cthreeo(N)
 &=\pthdo(N+1)-4\cthreeo(N)\notag\\
 &=\sum_{(k,\overline\pi)\in\mathcal B_{2,1}(2N)}
  (-1)^{\nu(\overline\pi)}.
\label{eq:three-color-boxed}
\end{align}
Burson's boxed $2$-modular involution cancels all nonfixed pairs.  A fixed
pair occurs exactly when $N=j(j+1)$, and is then
$(j,(2j+1)^j)$ with sign $(-1)^j$.
\end{corollary}

\begin{corollary}\label{cor:Pthdo-mod4}
Let $\chi_4$ be the nontrivial character modulo $4$.  For $n\ge0$,
\[
\pthdo(n)\equiv
\begin{cases}
\chi_4(x),&4n-3=x^2\text{ for some }x>0,\\
0,&4n-3\text{ is not a square}
\end{cases}
\pmod4.
\]
\end{corollary}

The cubic rank difference also gives the following congruence dissection.

\begin{corollary}
\label{cor:cubic-congruence-dissection}
If $N\equiv2\pmod6$, then
\begin{equation}\label{eq:Cto-cubic-mod3}
 \cthreeo(N)\equiv
 \begin{cases}
 (-1)^{m+1},&\substack{N=(3m+1)(3m+2)\\ \text{for some }m\ge0},\\
 0,&\text{otherwise}
 \end{cases}
 \pmod3.
\end{equation}
If $N\equiv3\pmod6$, then
\begin{equation}\label{eq:Pthdo-cubic-mod12}
 \pthdo(N)\equiv
 \begin{cases}
 5(-1)^{m+1},&\substack{N-1=(3m+1)(3m+2)\\ \text{for some }m\ge0},\\
 0,&\text{otherwise}
 \end{cases}
 \pmod{12}.
\end{equation}
\end{corollary}

\FloatBarrier
For reference, Table~\ref{tab:initial-values} lists the initial values of the
three less familiar counting functions and fixes the shifts used in the two
colored identities.
\begin{table}[!ht]
\centering
\caption{Initial values of the three less familiar counting functions.}
\label{tab:initial-values}
\begin{tabular}{c@{\quad}rrrrrrrr}
\toprule
$n$&0&1&2&3&4&5&6&7\\
\midrule
$\cthreeo(n)$&0&1&2&4&7&12&21&34\\
$\ptdo(n)$&0&1&3&4&8&12&17&28\\
$\pthdo(n)$&0&1&4&7&16&28&48&85\\
\bottomrule
\end{tabular}
\end{table}
\FloatBarrier

\section{Arithmetic of the sparse supports}
\label{sec:arithmetic}

The congruences in this section arise by excluding the support of the sparse
theta corrections.  We retain the notation
\[
T_j=\frac{j(j+1)}2
\]
and, for $M\ge2$, put
\begin{align*}
S_M&=\{j(j+1)\pmod M:j\ge0\},\\
R_M&=\{T_{2a+1}+T_b\pmod M:a,b\ge0\}.
\end{align*}
The proofs are given in Section~\ref{sec:proofs-arithmetic}.

We first isolate the two support-exclusion mechanisms used below.

\begin{theorem}\label{thm:Cto-support-congruences}
Let $M\ge2$ and $0\le r<M$.
\begin{enumerate}
\item If $r\notin R_M$, then, for every $n\ge0$,
\[
\cthreeo(Mn+r)\equiv0\pmod2,
\qquad
\pthdo(Mn+r+1)\equiv0\pmod8.
\]
\item If $3\mid M$, $r\equiv2\pmod3$, and $r\notin S_M$, then, for every
$n\ge0$,
\[
\cthreeo(Mn+r)\equiv0\pmod3,
\qquad
\pthdo(Mn+r+1)\equiv0\pmod{12}.
\]
\end{enumerate}
If both hypotheses hold, then, for every $n\ge0$,
\[
 \cthreeo(Mn+r)\equiv0\pmod6,
 \qquad
 \pthdo(Mn+r+1)\equiv0\pmod{24}.
\]
\end{theorem}

We now give the complete local description of the binary triangular
support.

\begin{theorem}\label{thm:all-modulus-support}
Let $M\ge2$ and $0\le r<M$.  Then
\[
\begin{aligned}
r\notin R_M
\quad\Longleftrightarrow\quad
&\text{there exists }p^\alpha\parallel M,\ p\equiv3\pmod4,\\
&\text{such that }v_p(4r+1)\text{ is odd and }v_p(4r+1)<\alpha.
\end{aligned}
\]
In particular, $R_M$ is the full residue system if and only if every prime
$p\equiv3\pmod4$ occurs in $M$ with exponent at most $1$.
\end{theorem}

The preceding two theorems give the following concrete prime-power family.

\begin{corollary}\label{cor:Cto-prime-powers}
Let $p\equiv3\pmod4$ be prime, $\alpha\ge2$, and
$0\le r<p^\alpha$.  If $v_p(4r+1)$ is odd and less than $\alpha$, then,
for every $n\ge0$,
\[
\cthreeo(p^\alpha n+r)\equiv0\pmod2,
\qquad
\pthdo(p^\alpha n+r+1)\equiv0\pmod8.
\]
For $p=3$, the moduli strengthen to $6$ and $24$, respectively.
\end{corollary}

For example, taking $p=3$, $\alpha=2$, and $r\in\{5,8\}$ gives
\[
\cthreeo(9n+r)\equiv0\pmod6,
\qquad
\pthdo(9n+r+1)\equiv0\pmod{24}
\qquad(n\ge0).
\]

All proofs for this section are collected in
Section~\ref{sec:proofs-arithmetic}.

\section{Proofs of the results in Section 4}
\label{sec:proofs-ordinary}

We give the proofs in the order in which the results are stated.

\begin{proof}[Proof of Proposition~\ref{prop:C-fTheta}]
In Ramanujan's identity \eqref{eq:ram-part2-1}, put $a=b=-1$.
The left-hand side then becomes
\[
 \sum_{n\ge0}\frac{q^n}{(q;q)_n^2}=C(q).
\]
Moreover, the coefficient $1+1/a$ of the second series is zero, so that
series disappears.  In the remaining series we have
\[
 (-1)^{m+1}a^{-m-1}b^m
 =(-1)^{m+1}(-1)^{-m-1}(-1)^m
 =(-1)^m,
\]
and
\[
 (-aq;q)_\infty(-bq;q)_\infty=(q;q)_\infty^2.
\]
Consequently,
\[
 C(q)=\frac{1}{(q;q)_\infty^2}
 \sum_{m\ge0}(-1)^mq^{\binom{m+1}{2}}
 =\frac{\fTheta(q)}{(q;q)_\infty^2},
\]
which is \eqref{eq:C-fTheta}.
\end{proof}

\begin{proof}[Proof of Theorem~\ref{thm:ordinary-colored}]
Let $L(N)$ denote the number of two-color partitions of $N$ in which the
largest part occurs exactly once and that occurrence is red.  If the largest
part is $M$, then it contributes $q^M$, while all remaining red and blue
parts are at most $M-1$.  Hence
\begin{align*}
 \sum_{N\ge1}L(N)q^N
 &=\sum_{M\ge1}\frac{q^M}{(q;q)_{M-1}^2} \\
 &=q\sum_{m\ge0}\frac{q^m}{(q;q)_m^2}
 =qC(q).
\end{align*}
Comparison of coefficients gives $L(n+1)=c(n)$.

For the minimal-excludant interpretation, pair consecutive terms in the
partial theta series:
\begin{align*}
 \fTheta(q)
 &=\sum_{j\ge0}
   \left(q^{\binom{2j+1}{2}}-q^{\binom{2j+2}{2}}\right) \\
 &=\sum_{j\ge0}q^{1+2+\cdots+2j}(1-q^{2j+1}).
\end{align*}
Using Proposition~\ref{prop:C-fTheta}, we therefore obtain
\begin{align*}
 C(q)
 &=\frac{1}{(q;q)_\infty}
   \sum_{j\ge0}
   \frac{q^{1+2+\cdots+2j}(1-q^{2j+1})}{(q;q)_\infty} \\
 &=\frac{1}{(q;q)_\infty}
   \sum_{j\ge0}
   \left(
   \prod_{k=1}^{2j}\frac{q^k}{1-q^k}
   \right)
   \left(
   \prod_{k\ge2j+2}\frac{1}{1-q^k}
   \right).
\end{align*}
The first factor $1/(q;q)_\infty$ generates the red parts.  For the blue
parts, the summand indexed by $j$ requires each of
$1,2,\ldots,2j$ to occur, forbids $2j+1$, and leaves every larger part
unrestricted.  Thus the smallest positive integer absent as a blue part is
$2j+1$, which is odd.  This proves the third interpretation; see also
Andrews and Newman \cite{AN}.
\end{proof}

\begin{proof}[Proof of Proposition~\ref{prop:c-parity}]
Work in the ring $\mathbb F_2[[q]]$.  Since
\[
 (1-q^j)^2\equiv1-q^{2j}\pmod2,
\]
we have $(q;q)_m^2\equiv(q^2;q^2)_m\pmod2$, and hence
\begin{align*}
 C(q)
 &=\sum_{m\ge0}\frac{q^m}{(q;q)_m^2}
 \equiv\sum_{m\ge0}\frac{q^m}{(q^2;q^2)_m}\pmod2.
\end{align*}
The $q$-binomial theorem \eqref{eq:qbinomial}, with base $q^2$, $a=0$,
and $z=q$, gives
\[
 \sum_{m\ge0}\frac{q^m}{(q^2;q^2)_m}
 =\frac{1}{(q;q^2)_\infty}.
\]
Since
\[
 \frac{1}{(q;q^2)_\infty}
 =\frac{(q^2;q^2)_\infty}{(q;q)_\infty}
 =(-q;q)_\infty
 \equiv(q;q)_\infty\pmod2,
\]
Euler's pentagonal number theorem \eqref{eq:euler-pentagonal} yields
\[
 C(q)\equiv
 \sum_{k=-\infty}^{\infty}q^{k(3k+1)/2}\pmod2.
\]
The exponents in this bilateral sum are precisely the generalized
pentagonal numbers $k(3k\pm1)/2$ with $k\ge0$.  The result follows by
comparing coefficients.
\end{proof}

\begin{proof}[Proof of Corollary~\ref{cor:c-residue}]
Suppose that $c(Mn+r)$ is odd and put $N=Mn+r$.  By
Proposition~\ref{prop:c-parity}, either
$N=k(3k-1)/2$ or $N=k(3k+1)/2$ for some $k\ge0$.  Accordingly,
\[
 24N+1=(6k-1)^2
 \qquad\text{or}\qquad
 24N+1=(6k+1)^2.
\]
Reducing modulo $M$ shows that $24r+1$ must be a square modulo $M$.
The contrapositive gives the asserted congruence.
\end{proof}

\begin{proof}[Proof of Proposition~\ref{prop:Cf}]
Set $a=b=1$ in \eqref{eq:ram-part2-1}.  The left-hand side is
$C_f(q)$.  The first series on the right becomes
\[
 -\frac{1}{(-q;q)_\infty^2}
  \sum_{m\ge0}(-1)^mq^{\binom{m+1}{2}}
 =-\frac{\fTheta(q)}{(-q;q)_\infty^2}.
\]
The factor $1+1/a$ in the second series is now $2$, and the remaining
series is exactly the left-hand side of \eqref{eq:ram-ch9-2}.  Therefore
\begin{align*}
 C_f(q)
 &=-\frac{\fTheta(q)}{(-q;q)_\infty^2}
   +2\sum_{m\ge0}
     \frac{(-1)^mq^{m(m+1)/2}}{(-q;q)_m} \\
 &=-\frac{\fTheta(q)}{(-q;q)_\infty^2}+2\fE(q),
\end{align*}
which proves \eqref{eq:Cf}.  In particular, the negative sign in the first
term follows directly from the specialization.
\end{proof}

\begin{proof}[Proof of Proposition~\ref{prop:Cphi}]
Since
\[
 \frac{1}{(-q^2;q^2)_n}
 =\frac{(q^2;q^2)_n}{(q^4;q^4)_n},
\]
we may write
\begin{align*}
 C_\phi(q)
 &=\sum_{n\ge0}\frac{(q^2;q^2)_nq^n}{(q^4;q^4)_n} \\
 &=(q^2;q^2)_\infty
   \sum_{n\ge0}
   \frac{q^n}{(q^4;q^4)_n(q^{2n+2};q^2)_\infty}.
\end{align*}
By the $q$-binomial theorem \eqref{eq:qbinomial} with base $q^2$,
\[
 \frac{1}{(q^{2n+2};q^2)_\infty}
 =\sum_{m\ge0}\frac{q^{(2n+2)m}}{(q^2;q^2)_m}.
\]
Substitution and interchange of the two sums give
\begin{align*}
 C_\phi(q)
 &=(q^2;q^2)_\infty
   \sum_{m\ge0}\frac{q^{2m}}{(q^2;q^2)_m}
   \sum_{n\ge0}\frac{q^{(2m+1)n}}{(q^4;q^4)_n} \\
 &=(q^2;q^2)_\infty
   \sum_{m\ge0}
   \frac{q^{2m}}{(q^2;q^2)_m(q^{2m+1};q^4)_\infty},
\end{align*}
where the $q$-binomial theorem \eqref{eq:qbinomial} with base $q^4$ was used once more.

We now separate the even and odd values of $m$.  The product identities
\begin{align*}
 (q^2;q^2)_{2r}&=(q^2;q^4)_r(q^4;q^4)_r,\\
 (q^2;q^2)_{2r+1}&=(1-q^2)(q^6;q^4)_r(q^4;q^4)_r,\\
 \frac{1}{(q^{4r+1};q^4)_\infty}
 &=\frac{(q;q^4)_r}{(q;q^4)_\infty},\\
 \frac{1}{(q^{4r+3};q^4)_\infty}
 &=\frac{(q^3;q^4)_r}{(q^3;q^4)_\infty}
\end{align*}
show that $C_\phi(q)=E(q)+O(q)$, where
\begin{align*}
 E(q)&=\frac{(q^2;q^2)_\infty}{(q;q^4)_\infty}
 \sum_{r\ge0}
 \frac{(q;q^4)_rq^{4r}}{(q^4;q^4)_r(q^2;q^4)_r},\\
 O(q)&=\frac{q^2(q^2;q^2)_\infty}
 {(1-q^2)(q^3;q^4)_\infty}
 \sum_{r\ge0}
 \frac{(q^3;q^4)_rq^{4r}}{(q^4;q^4)_r(q^6;q^4)_r}.
\end{align*}
Apply Heine's transformation \eqref{eq:heine} with base $q^4$.  Taking
$(a,b,c,z)=(q,q^4,0,q)$ gives
\[
 E(q)=\sum_{n\ge0}(q;q^4)_nq^n,
\]
while taking $(a,b,c,z)=(q^3,q^4,0,q^3)$ gives
\[
 O(q)=q^2\sum_{n\ge0}(q^3;q^4)_nq^{3n}.
\]
Here we have used
\[
 (q^2;q^2)_\infty=(q^2;q^4)_\infty(q^4;q^4)_\infty
\]
and
\[
 \frac{(q^2;q^2)_\infty}{1-q^2}
 =(q^6;q^4)_\infty(q^4;q^4)_\infty.
\]
Thus
\[
 C_\phi(q)
 =\sum_{n\ge0}(q;q^4)_nq^n
  +q^2\sum_{n\ge0}(q^3;q^4)_nq^{3n},
\]
which proves \eqref{eq:Cphi-finite}.

It remains to identify this expression with $\fE(-q)$.  Applying the
Rogers--Fine formula \eqref{eq:rogers-fine} first with base $q^4$ and
$a=q$, and then with base $q^4$ and $a=q^3$, gives
\begin{align*}
 \sum_{n\ge0}(q;q^4)_nq^n
 &=\sum_{r\ge0}(-1)^rq^{6r^2+r}(1+q^{4r+1}),\\
 q^2\sum_{n\ge0}(q^3;q^4)_nq^{3n}
 &=\sum_{r\ge0}(-1)^rq^{6r^2+7r+2}(1+q^{4r+3}).
\end{align*}
On the other hand, split the defining series
\[
 \fE(-q)=\sum_{n\ge0}(-q)^{n(3n+1)/2}
          \bigl(1-(-q)^{2n+1}\bigr)
\]
into its even and odd indices.  For $n=2r$ the summand is
\[
 (-1)^rq^{6r^2+r}(1+q^{4r+1}),
\]
whereas for $n=2r+1$ it is
\[
 (-1)^rq^{6r^2+7r+2}(1+q^{4r+3}).
\]
These are exactly the two Rogers--Fine series displayed above.  Hence
$C_\phi(q)=\fE(-q)$, proving \eqref{eq:Cphi-fE}.
\end{proof}

\section{Proofs of the results in Section 5}
\label{sec:proofs-odd}

We give the proofs in the order in which the results are stated.  The proofs
of Theorems~\ref{thm:Ptdo-positive-decomposition} and~\ref{thm:Co-Ptdo}
use the same parity splitting, which is isolated below together with the
partial-theta evaluation needed for the discrepancy.

\begin{proof}[Proof of Proposition~\ref{prop:odd-boxed-correction}]
By \eqref{eq:burson-boxed},
\[
 \mathfrak B_{3,1}(w^3,q)
 =\sum_{j\ge0}(-1)^jw^{6j}q^{3j^2+2j}
\]
and
\[
 \mathfrak B_{3,2}(w^3,q)
 =\sum_{j\ge0}(-1)^jw^{6j}q^{3j^2+4j}.
\]
Multiplying the first identity by $-w^2$ and the second by $-w^4q$, and
then adding, gives \eqref{eq:odd-boxed-correction}.
\end{proof}

\begin{proof}[Proof of Theorem~\ref{thm:Co-colored}]
Let $L_o(N)$ denote the number of two-color partitions of $N$ into odd parts
in which the largest part occurs exactly once, and that occurrence is red.
If the largest part is $2m+3$, it contributes $q^{2m+3}$, while all remaining
red and blue parts are odd and at most $2m+1$.  Hence
\begin{align*}
\sum_{N\ge2}L_o(N)q^N
&=\sum_{m\ge0}\frac{q^{2m+3}}{(q;q^2)_{m+1}^2} \\
&=q^2\sum_{m\ge0}\frac{q^{2m+1}}{(q;q^2)_{m+1}^2}
=q^2C_\omega(q).
\end{align*}
Comparison of coefficients gives $L_o(n+2)=\co(n)$ for every $n\ge0$.
\end{proof}

To prove Theorems~\ref{thm:Ptdo-positive-decomposition} and
\ref{thm:Co-Ptdo}, we must separate the auxiliary series according to the
parity of an intermediate index.  The odd-index part will give the
odd-corner series $C_\omega(q)$, while the even-index part will give the
staircase-decorated series $E_\triangle(q)$.  Their sum yields the positive
decomposition and their difference yields the sparse correction.  The next
two lemmas carry out these two steps.

\begin{lemma}\label{lem:Co-split}
Put
\[
T_1(q)=(-q;q)_\infty^2
\sum_{m\ge0}\frac{\qpoch{q}{q^2}{m}q^m}{\qpoch{-q}{q}{m}},
\]
and
\[
T_2(q)=\qinf{q^2}{q^2}
\sum_{m\ge0}\frac{\qpoch{q}{q^2}{m}q^m}{\qpoch{q}{q}{m}}.
\]
Then
\begin{equation}\label{eq:T1-positive-split}
 T_1(q)=C_\omega(q)+E_\triangle(q),
\end{equation}
and
\begin{equation}\label{eq:T2-signed-split}
 T_2(q)=E_\triangle(q)-C_\omega(q).
\end{equation}
In particular,
\[
 C_\omega(q)=\frac12T_1(q)-\frac12T_2(q).
\]
\end{lemma}

\begin{proof}
For $r\ge0$, define
\begin{equation}\label{eq:Hr-definition}
 H_r(q):=
 \frac{\qinf{q^2}{q^2}}{\qinf{q}{q^2}}
 \frac{q^r}{\qpoch{q}{q}{r}}
 \sum_{m\ge0}
 \frac{\qpoch{q}{q^2}{m}q^{(r+1)m}}
      {\qpoch{q^2}{q^2}{m}}.
\end{equation}
Interchanging the two sums and using the $q$-binomial theorem
\eqref{eq:qbinomial} first with $z=q^{m+1}$ and then with
$z=-q^{m+1}$ gives
\begin{equation}\label{eq:T1-T2-Hr}
 T_1(q)=\sum_{r\ge0}H_r(q),
 \qquad
 T_2(q)=\sum_{r\ge0}(-1)^rH_r(q).
\end{equation}
Indeed, the two inner sums over $r$ are
\[
 \sum_{r\ge0}\frac{q^{r(m+1)}}{\qpoch{q}{q}{r}}
 =\frac{1}{\qinf{q^{m+1}}{q}},
 \qquad
 \sum_{r\ge0}\frac{(-1)^rq^{r(m+1)}}{\qpoch{q}{q}{r}}
 =\frac{1}{\qinf{-q^{m+1}}{q}},
\]
and the product identities
\[
 \qpoch{q^2}{q^2}{m}=\qpoch{q}{q}{m}\qpoch{-q}{q}{m},
 \qquad
 \qinf{-q}{q}=\frac1{\qinf{q}{q^2}}
\]
reduce the resulting expressions to $T_1(q)$ and $T_2(q)$.

A further application of \eqref{eq:qbinomial}, now with base $q^2$, yields
\[
 \sum_{m\ge0}
 \frac{\qpoch{q}{q^2}{m}q^{(r+1)m}}
      {\qpoch{q^2}{q^2}{m}}
 =\frac{\qinf{q^{r+2}}{q^2}}{\qinf{q^{r+1}}{q^2}}.
\]
If $r=2k+1$, then
\begin{align*}
 H_{2k+1}(q)
 &=q^{2k+1}
 \frac{\qpoch{q^2}{q^2}{k}}
      {\qpoch{q}{q^2}{k+1}\qpoch{q}{q}{2k+1}}\\
 &=\frac{q^{2k+1}}{\qpoch{q}{q^2}{k+1}^2},
\end{align*}
because
$\qpoch{q}{q}{2k+1}=\qpoch{q}{q^2}{k+1}\qpoch{q^2}{q^2}{k}$.
Thus
\begin{equation}\label{eq:H-odd}
 H_{2k+1}(q)=
 \frac{q^{2k+1}}{\qpoch{q}{q^2}{k+1}^2},
\end{equation}
which is the summand of $C_\omega(q)$ with peak $2k+1$.

If $r=2k$, then
\begin{align*}
 H_{2k}(q)
 &=q^{2k}
 \frac{\qinf{q^2}{q^2}^2}{\qinf{q}{q^2}^2}
 \frac{\qpoch{q}{q^2}{k}}
      {\qpoch{q^2}{q^2}{k}\qpoch{q}{q}{2k}}\\
 &=G(q)^2\frac{q^{2k}}{\qpoch{q^2}{q^2}{k}^2},
\end{align*}
where we used \eqref{eq:gauss} and
$\qpoch{q}{q}{2k}=\qpoch{q}{q^2}{k}\qpoch{q^2}{q^2}{k}$.
Therefore
\begin{equation}\label{eq:H-even}
 H_{2k}(q)=
 G(q)^2\frac{q^{2k}}{\qpoch{q^2}{q^2}{k}^2},
\end{equation}
which is the summand of $G(q)^2C(q^2)=E_\triangle(q)$ with doubled peak
$2k$.  Thus the odd-index part of \eqref{eq:T1-T2-Hr} is $C_\omega(q)$
and the even-index part is $E_\triangle(q)$.  Adding and subtracting these
two parts proves \eqref{eq:T1-positive-split} and
\eqref{eq:T2-signed-split}.
\end{proof}

\begin{lemma}\label{lem:T2}
We have
\[
T_2(q)=\sum_{j\ge0}(-1)^jq^{3j^2+2j}\left(1+q^{2j+1}\right).
\]
\end{lemma}

\begin{proof}
We use Heine's transformation \eqref{eq:heine}, together with the formal
half-power convention stated in Section~\ref{sec:rank-functions}.  Since
\[
\qpoch{q}{q^2}{m}=\qpoch{q^{1/2}}{q}{m}\qpoch{-q^{1/2}}{q}{m},
\]
we apply \eqref{eq:heine} with
\[
(a,b,c,z)=(-q^{1/2},q^{1/2},0,q).
\]
This gives
\begin{align*}
T_2(q)
&=\qinf{q^2}{q^2}\sum_{m\ge0}
\frac{\qpoch{q}{q^2}{m}q^m}{\qpoch{q}{q}{m}}  \\
&=\frac{\qinf{q^2}{q^2}\qinf{q^{1/2}}{q}\qinf{-q^{1/2}}{q}}
        {\qinf{q}{q}}
\sum_{n\ge0}\frac{q^{n/2}}{\qpoch{-q^{1/2}}{q}{n+1}}.
\end{align*}
The prefactor is $1$, because
\[
\qinf{q^{1/2}}{q}\qinf{-q^{1/2}}{q}=\qinf{q}{q^2}
\quad\text{and}\quad
\qinf{q}{q}=\qinf{q}{q^2}\qinf{q^2}{q^2}.
\]
Replacing $q$ by $q^{1/2}$ in \eqref{eq:ram-ch9-3} now gives
\[
\sum_{n\ge0}\frac{q^{n/2}}{\qpoch{-q^{1/2}}{q}{n+1}}
=\sum_{j\ge0}(-1)^jq^{3j^2+2j}\left(1+q^{2j+1}\right),
\]
which proves the lemma.
\end{proof}

\begin{proof}[Proof of Theorem~\ref{thm:Ptdo-positive-decomposition}]
The generating function \eqref{eq:Ptdo-gf} gives
\begin{align*}
\sum_{n\ge0}\ptdo(n)q^n
&=q\sum_{m\ge0}q^m
\frac{\qinf{-q^{m+1}}{q}}{\qinf{q^{2m+1}}{q^2}}  \\
&=q\frac{\qinf{-q}{q}}{\qinf{q}{q^2}}
\sum_{m\ge0}\frac{\qpoch{q}{q^2}{m}q^m}{\qpoch{-q}{q}{m}} \\
&=qT_1(q).
\end{align*}
Since $\ptdo(0)=0$, division by $q$ and
\eqref{eq:T1-positive-split} give
\[
 \sum_{N\ge0}\ptdo(N+1)q^N
 =T_1(q)=C_\omega(q)+E_\triangle(q).
\]
This proves \eqref{eq:Ptdo-positive-decomposition}; coefficient comparison
gives \eqref{eq:Ptdo-positive-coeff}.
\end{proof}

\begin{proof}[Proof of Theorem~\ref{thm:Co-Ptdo}]
The preceding calculation gives
\[
 \sum_{N\ge0}\ptdo(N+1)q^N=T_1(q).
\]
By \eqref{eq:T1-positive-split} and \eqref{eq:T2-signed-split},
$T_1(q)-2C_\omega(q)=T_2(q)$.  Lemma~\ref{lem:T2} therefore gives
\[
\sum_{N\ge0}\left(\ptdo(N+1)-2\co(N)\right)q^N
=
\sum_{j\ge0}(-1)^jq^{3j^2+2j}\left(1+q^{2j+1}\right).
\]
The theorem follows by comparing coefficients.
\end{proof}

\begin{proof}[Proof of Corollary~\ref{cor:two-color-boxed}]
The positive decomposition \eqref{eq:Ptdo-positive-decomposition} gives
\begin{align}
 E_\triangle(q)-C_\omega(q)
 &=\sum_{N\ge0}
   \bigl(\ptdo(N+1)-2\co(N)\bigr)q^N.                 \label{eq:proof-cor55-first}
\end{align}
By Theorem~\ref{thm:Co-Ptdo}, the right-hand side of
\eqref{eq:proof-cor55-first} is
\begin{equation}\label{eq:proof-cor55-sparse}
 \sum_{j\ge0}(-1)^jq^{3j^2+2j}
 +\sum_{j\ge0}(-1)^jq^{3j^2+4j+1}.
\end{equation}

We now identify the two series in \eqref{eq:proof-cor55-sparse} with the two
boxed modular series.  Setting $u=1$ in \eqref{eq:burson-boxed}, first with
$(m,r)=(3,1)$ and then with $(m,r)=(3,2)$, gives
\[
 \mathfrak B_{3,1}(1,q)
 =\sum_{j\ge0}(-1)^jq^{j(3j+2)}
 =\sum_{j\ge0}(-1)^jq^{3j^2+2j}
\]
and
\[
 q\mathfrak B_{3,2}(1,q)
 =q\sum_{j\ge0}(-1)^jq^{j(3j+4)}
 =\sum_{j\ge0}(-1)^jq^{3j^2+4j+1}.
\]
Consequently,
\begin{equation}\label{eq:proof-cor55-B-sum}
 E_\triangle(q)-C_\omega(q)
 =\mathfrak B_{3,1}(1,q)+q\mathfrak B_{3,2}(1,q).
\end{equation}

The first line of \eqref{eq:burson-boxed} shows that
\[
 [q^N]\mathfrak B_{3,1}(1,q)
 =\sum_{(k,\overline\pi)\in\mathcal B_{3,1}(N)}
  (-1)^{\nu(\overline\pi)}.
\]
Similarly,
\begin{align*}
 [q^N]\,q\mathfrak B_{3,2}(1,q)
 &=[q^{N-1}]\mathfrak B_{3,2}(1,q)\\
 &=\sum_{(k,\overline\pi)\in\mathcal B_{3,2}(N-1)}
   (-1)^{\nu(\overline\pi)},
\end{align*}
where the last sum is empty for $N=0$.  Taking the coefficient of $q^N$ in
\eqref{eq:proof-cor55-B-sum} proves \eqref{eq:two-color-boxed}.

It remains to explain the exceptional weights.  For $(m,r)=(3,1)$, the fixed
pair in \eqref{eq:burson-fixed-pairs} is
$(j,(3j+1)^j)$.  Its sign is $(-1)^j$, and its weight is
\[
 1\cdot j+j(3j+1)=3j^2+2j.
\]
For $(m,r)=(3,2)$, the fixed pair is $(j,(3j+2)^j)$, with sign $(-1)^j$
and weight
\[
 2j+j(3j+2)=3j^2+4j.
\]
The additional factor $q$ in \eqref{eq:proof-cor55-B-sum} shifts this second
weight to $3j^2+4j+1$.  Burson's involutions cancel all remaining pairs,
which proves the final assertion.
\end{proof}

\begin{proof}[Proof of Proposition~\ref{prop:Co-parity}]
From \eqref{eq:Co-gf},
\begin{align*}
\sum_{n\ge0}\co(n)q^n
&=\sum_{m\ge0}\frac{q^{2m+1}}{\qpoch{q}{q^2}{m+1}^2}  \\
&\equiv
q\sum_{m\ge0}\frac{q^{2m}}{\qpoch{q^2}{q^4}{m+1}}
\pmod2.
\end{align*}
Modulo $2$, the signs in \eqref{eq:ram-ch9-3} do not matter.  Replacing
$q$ there by $q^2$ therefore gives
\[
\sum_{n\ge0}\co(n)q^n
\equiv
\sum_{j\ge0}q^{12j^2+8j+1}
+
\sum_{j\ge0}q^{12j^2+16j+5}
\pmod2.
\]
The two sequences do not overlap.  Moreover,
\[
3(12j^2+8j+1)+1=(6j+2)^2,
\]
and
\[
3(12j^2+16j+5)+1=(6j+4)^2.
\]
The first sequence has positive square root $6j+2$, and the second has
positive square root $6j+4$.  Conversely, suppose that $3n+1=x^2$ with $x$
even.  Since $x^2\equiv1\pmod3$, the integer $x$ is not divisible by $3$;
hence $x\equiv2$ or $4\pmod6$.  Thus $x=6j+2$ or $x=6j+4$ for some
$j\ge0$, and $n$ belongs to one of the two displayed exponent sequences.
Therefore $\co(n)$ is odd exactly when $3n+1$ is an even square.
\end{proof}

\begin{proof}[Proof of Corollary~\ref{cor:Ptdo-mod4}]
The case $N=0$ is immediate from $\ptdo(0)=0$.  Let $N\ge1$ and put
$n=N-1$ in \eqref{eq:Ptdo-theta-coeff}.  Then
\begin{equation}\label{eq:proof-cor57-start}
 \ptdo(N)=2\co(N-1)+s(N-1),
\end{equation}
where
\[
 s(n)=[q^n]\sum_{j\ge0}(-1)^jq^{3j^2+2j}(1+q^{2j+1}).
\]
Thus $s(n)$ is nonzero only at the two families
\[
 A_j=3j^2+2j,
 \qquad
 B_j=3j^2+4j+1,
\]
and in either family its value is $(-1)^j$.

We first relate these two families to the square condition in the statement.
A direct calculation gives
\begin{equation}\label{eq:proof-cor57-squares}
 3(A_j+1)-2=(3j+1)^2,
 \qquad
 3(B_j+1)-2=(3j+2)^2.
\end{equation}
Conversely, suppose that $3N-2=x^2$ for some $x>0$.  Since
$x^2\equiv1\pmod3$, one has $3\nmid x$.  Hence either $x=3j+1$ or
$x=3j+2$ for a unique $j\ge0$.  The two alternatives in
\eqref{eq:proof-cor57-squares} then show that, respectively,
$N-1=A_j$ or $N-1=B_j$.  Therefore the square condition is equivalent to
$N-1$ belonging to one of the two supports of $s$.

Assume first that $N-1=A_j$.  Proposition~\ref{prop:Co-parity} applies to
\[
 3A_j+1=(3j+1)^2.
\]
The number $3j+1$ is even exactly when $j$ is odd.  Hence
\[
 2\co(A_j)\equiv
 \begin{cases}
 0,&j\text{ even},\\
 2,&j\text{ odd}
 \end{cases}
 \pmod4.
\]
Using \eqref{eq:proof-cor57-start} and $s(A_j)=(-1)^j$, we obtain in both
parity cases
\[
 \ptdo(A_j+1)\equiv1\pmod4.
\]
The positive square root in this case is $x=3j+1$, and therefore
$\chi_3(x)=1$.

Now assume that $N-1=B_j$.  Since
\[
 3B_j+1=(3j+2)^2,
\]
Proposition~\ref{prop:Co-parity} shows that $\co(B_j)$ is odd exactly when
$3j+2$ is even, or equivalently when $j$ is even.  Thus
\[
 2\co(B_j)\equiv
 \begin{cases}
 2,&j\text{ even},\\
 0,&j\text{ odd}
 \end{cases}
 \pmod4.
\]
Together with $s(B_j)=(-1)^j$, equation \eqref{eq:proof-cor57-start} gives
\[
 \ptdo(B_j+1)\equiv-1\pmod4.
\]
Here the positive square root is $x=3j+2$, so $\chi_3(x)=-1$.

Finally, if $3N-2$ is not a square, then $N-1$ is neither $A_j$ nor $B_j$;
hence $s(N-1)=0$.  Proposition~\ref{prop:Co-parity} also shows that
$\co(N-1)$ is even, because $3(N-1)+1=3N-2$ is not an even square.
Equation \eqref{eq:proof-cor57-start} therefore gives
$\ptdo(N)\equiv0\pmod4$.  This proves the corollary.
\end{proof}

\begin{proof}[Proof of Corollary~\ref{cor:two-nonsquare}]
If $N=Mn+r$ and $3r+1$ is a nonsquare modulo $M$, then $3N+1$ cannot be an
integer square, so Proposition~\ref{prop:Co-parity} gives the congruence for
$\co$.  If $3r-2$ is a nonsquare modulo $M$, then $3N-2$ cannot be an
integer square, so Corollary~\ref{cor:Ptdo-mod4} gives the congruence for
$\ptdo$.
\end{proof}

\begin{proof}[Proof of Proposition~\ref{prop:Cnu}]
Ramanujan's identity \eqref{eq:ram-ch9-3} gives
\begin{align*}
C_\nu(q)
&=\sum_{m\ge0}\frac{q^m}{(-q;q^2)_{m+1}} \\
&=\sum_{m\ge0}(-1)^m q^{6m^2+4m}(1+q^{4m+2})
=\fO(-q^2).
\end{align*}
It remains to identify the finite-product series with the same false theta
function.  Apply the Rogers--Fine formula \eqref{eq:rogers-fine} with base
$q^4$ and $a=q^2$.  Then
\begin{align*}
\sum_{m\ge0}(q^2;q^4)_mq^{2m}
&=\sum_{m\ge0}(-1)^m(q^4)^{m(3m-1)/2}q^{6m}
   (1+q^{4m+2}) \\
&=\sum_{m\ge0}(-1)^m q^{6m^2+4m}(1+q^{4m+2})
=\fO(-q^2).
\end{align*}
Thus both series in \eqref{eq:Cnu} are equal to $\fO(-q^2)$.
\end{proof}

\begin{proof}[Proof of Proposition~\ref{prop:Cpsi}]
By \eqref{eq:Cpsi-def}, $C_\psi(q)$ is the left-hand side of
\eqref{eq:ram-part3-psi}, while \eqref{eq:omega3} identifies its right-hand
side as $q\omega_3(q)$.  This proves \eqref{eq:Cpsi}.
\end{proof}

\section{Proofs of the results in Section 6}
\label{sec:proofs-three}

We first record the boxed modular application, then prove the exceptional
parity results and the quartic and cubic root-of-unity identities.  The proof
of Theorem~\ref{thm:Cto-Pthdo} is parallel to that of
Theorem~\ref{thm:Co-Ptdo}: we isolate an auxiliary series and evaluate it
using \eqref{eq:heine} and \eqref{eq:ram-ch9-1}.  The final corollaries
follow from the resulting sparse theta terms and the cubic rank difference.

\begin{proof}[Proof of Proposition~\ref{prop:three-boxed-correction}]
In \eqref{eq:burson-boxed}, take $(m,r)=(2,1)$ and $u=w$, and replace
the base $q$ by $q^{1/2}$.  Its last member becomes
\[
 \sum_{j\ge0}(-1)^jw^{2j}
 (q^{1/2})^{j(2j+2)}
 =\sum_{j\ge0}(-w^2)^jq^{j(j+1)},
\]
which is \eqref{eq:three-boxed-correction}.  The fixed-point assertion is
part of Burson's construction recorded in \eqref{eq:burson-fixed-pairs}.
\end{proof}

\begin{proof}[Proof of Theorem~\ref{thm:D-parity-specialization}]
Write
\[
 A(z;q)=\sum_{j\ge0}(-z)^jq^{j(j+1)/2},
 \qquad
 B(z;q)=\sum_{j\ge0}(-z)^jq^{j(j+1)},
\]
and
\[
 Q(z;q)=(zq;q^2)_\infty(z^{-1}q;q^2)_\infty.
\]
Then \eqref{eq:Dz-ramanujan} is
\[
 (1+z^{-1})\mathcal D(z;q)
 =\frac{A(z;q)}{(q;q^2)_\infty Q(z;q)}-B(z;q).
\]
Both sides are analytic in a neighborhood of $z=-1$.  Differentiating and
then setting $z=-1$ gives
\begin{equation}\label{eq:D-parity-derivative-step}
 -\mathcal D(-1;q)
 =\left.\frac{d}{dz}
 \left(\frac{A(z;q)}{(q;q^2)_\infty Q(z;q)}\right)\right|_{z=-1}
 -B'(-1;q).
\end{equation}
Now
\[
 A(-1;q)=G(q),
 \qquad A'(-1;q)=-G_1(q),
 \qquad B'(-1;q)=-G_1(q^2).
\]
Moreover, $Q(z;q)=Q(z^{-1};q)$, and therefore $Q'(-1;q)=0$.  Finally,
setting $z=-1$ in \eqref{eq:Dz-ramanujan} shows that
\[
 \frac{G(q)}{(q;q^2)_\infty Q(-1;q)}=G(q^2).
\]
Substitution in \eqref{eq:D-parity-derivative-step} yields
\[
 -\mathcal D(-1;q)
 =-\frac{G(q^2)}{G(q)}G_1(q)+G_1(q^2),
\]
which is \eqref{eq:D-parity-specialization}.  Finally, Gauss's identity
\eqref{eq:gauss} gives
\[
 \frac{G(q^2)}{G(q)}
 =\frac{(q;q^2)_\infty}{(q^2;q^4)_\infty^2},
\]
and hence \eqref{eq:D-parity-specialization-product}.
\end{proof}

\begin{proof}[Proof of Theorem~\ref{thm:D-parity-bias}]
Starting from \eqref{eq:Dz-gf}, set $z=-1$ and then replace $q$ by $-q$.
Since
\[
 (q;-q)_{2m}=(q;q^2)_m(-q^2;q^2)_m,
\]
we obtain
\begin{align*}
 -\mathcal D(-1;-q)
 &=\sum_{m\ge0}
 \frac{q^{2m+1}(q;-q)_{2m}}{(q;q^2)_{m+1}^2}\\
 &=\sum_{m\ge0}
 \frac{q^{2m+1}(-q^2;q^2)_m}
 {(q;q^2)_m(1-q^{2m+1})^2},
\end{align*}
which proves \eqref{eq:D-parity-positive}.  For each $m\ge0$, the polynomial $(-q^2;q^2)_m$ has nonnegative
coefficients, and so do the reciprocal products
$1/(q;q^2)_m$ and $1/(1-q^{2m+1})^2$.  Hence every summand on the
right-hand side of \eqref{eq:D-parity-positive} has nonnegative
coefficients.  More precisely, the summand with $m=0$ is
\[
 \frac{q}{(1-q)^2}=\sum_{n\ge1}nq^n.
\]
Therefore, for every $n\ge1$, the coefficient of $q^n$ in the full
right-hand side is at least $n$, and is consequently strictly positive.
On the other hand,
\[
 -\mathcal D(-1;-q)
 =\sum_{n\ge0}(-1)^{n+1}
 \bigl(d_0^{(2)}(n)-d_1^{(2)}(n)\bigr)q^n.
\]
Comparing the coefficient of $q^n$ for each $n\ge1$ proves
\eqref{eq:D-parity-inequality}.
\end{proof}

\begin{proof}[Proof of Theorem~\ref{thm:D-cyclotomic-differences}]
We first treat $z=i$.  Put
\[
 H_4(q):=(q;q^2)_\infty(-q^2;q^4)_\infty.
\]
Since
\[
 (iq;q^2)_\infty(-iq;q^2)_\infty=(-q^2;q^4)_\infty,
\]
identity \eqref{eq:Dz-ramanujan} gives
\begin{equation}\label{eq:D-i-A}
 \mathcal D(i;q)=\frac{1+i}{2}
 \left(\frac{A_i(q)}{H_4(q)}-A_i(q^2)\right),
 \qquad
 A_i(q):=\sum_{j\ge0}(-i)^j q^{T_j}.
\end{equation}
Define the bilateral theta series
\[
 \Theta_4(q):=\sum_{n\in\mathbb Z}(-1)^n q^{n(2n+1)}.
\]
Jacobi's triple product \eqref{eq:jacobi-triple}, with base $q^4$ and
$x=q^3$, gives
\begin{equation}\label{eq:Theta4-product}
 \Theta_4(q)=(q;q^2)_\infty(q^4;q^4)_\infty.
\end{equation}
Consequently,
\begin{equation}\label{eq:Theta4-functional}
 \Theta_4(q)=H_4(q)\Theta_4(q^2),
\end{equation}
because
\[
 (-q^2;q^4)_\infty(q^2;q^4)_\infty
 =(q^4;q^8)_\infty.
\]
Writing the negative terms in the bilateral sum with nonnegative indices,
we have
\[
 \Theta_4(q)=\sum_{n\ge0}(-1)^n
 \left(q^{T_{2n}}-q^{T_{2n+1}}\right).
\]
On the other hand, separating the even and odd indices gives
\[
 A_i(q)=\sum_{n\ge0}(-1)^n
 \left(q^{T_{2n}}-i q^{T_{2n+1}}\right).
\]
Combining these two expressions with the definition of $\Phi_4(q)$ gives
\begin{equation}\label{eq:Ai-decomposition}
 \frac{1+i}{2}A_i(q)
 =\Phi_4(q)+\frac{-1+i}{2}\Theta_4(q).
\end{equation}
Substituting \eqref{eq:Ai-decomposition} into \eqref{eq:D-i-A} and using
\eqref{eq:Theta4-functional}, the theta-product terms cancel, leaving
\eqref{eq:D-i-difference}.

We next take $z=\omega$, where $1+\omega+\omega^2=0$.  Since
\[
 (1-\omega x)(1-\omega^2x)=1+x+x^2=\frac{1-x^3}{1-x},
\]
we have
\begin{equation}\label{eq:cubic-product-collapse}
 (q;q^2)_\infty(\omega q;q^2)_\infty
 (\omega^2q;q^2)_\infty=(q^3;q^6)_\infty.
\end{equation}
Thus \eqref{eq:Dz-ramanujan} becomes
\begin{equation}\label{eq:D-omega-A}
 \mathcal D(\omega;q)=-\omega^2
 \left(\frac{A_\omega(q)}{(q^3;q^6)_\infty}
 -A_\omega(q^2)\right),
 \qquad
 A_\omega(q):=\sum_{j\ge0}(-\omega)^j q^{T_j}.
\end{equation}
Introduce
\[
 \Theta_3(q):=\sum_{n\in\mathbb Z}(-1)^n q^{T_{3n}}.
\]
Jacobi's triple product \eqref{eq:jacobi-triple}, with base $q^9$ and
$x=q^6$, gives
\begin{equation}\label{eq:Theta3-product}
 \Theta_3(q)=(q^3;q^3)_\infty
 =(q^3;q^6)_\infty\Theta_3(q^2).
\end{equation}
Moreover,
\[
 \Theta_3(q)=\sum_{n\ge0}(-1)^n
 \left(q^{T_{3n}}-q^{T_{3n+2}}\right).
\]
To compare $A_\omega(q)$ with these series, write $j=3m+r$.  Since
$(-\omega)^3=-1$, the three resulting coefficients are
\[
 -\omega^2(-\omega)^{3m}=-\omega^2(-1)^m,
 \qquad
 -\omega^2(-\omega)^{3m+1}=(-1)^m,
\]
and
\[
 -\omega^2(-\omega)^{3m+2}=-\omega(-1)^m.
\]
Using $-\omega^2=1+\omega$, and then grouping the terms with
$T_{3m}$, $T_{3m+1}$, and $T_{3m+2}$, gives
\begin{equation}\label{eq:Aomega-decomposition}
 -\omega^2 A_\omega(q)=\Phi_3(q)+\omega\Theta_3(q).
\end{equation}
Substitution in \eqref{eq:D-omega-A}, followed by
\eqref{eq:Theta3-product}, cancels the bilateral theta terms and proves
\eqref{eq:D-omega-difference}.

Finally, the symmetry $d(r,n)=d(-r,n)$ gives
\[
 \mathcal D(i;q)=\sum_{n\ge0}
 \left(d_0^{(4)}(n)-d_2^{(4)}(n)\right)q^n
\]
and, since $\omega+\omega^2=-1$,
\[
 \mathcal D(\omega;q)=\sum_{n\ge0}
 \left(d_0^{(3)}(n)-d_1^{(3)}(n)\right)q^n.
\]
\end{proof}

\begin{proof}[Proof of Corollary~\ref{cor:cubic-rank-equidistribution}]
The reciprocal of $(q^3;q^6)_\infty$ is supported on multiples of $3$.
The exponents $T_{3m}$ are congruent to $0$ or $3$ modulo $6$, and the
exponents $T_{3m+1}$ are congruent to $1$ or $4$ modulo $6$.  Hence the first
term on the right of \eqref{eq:D-omega-difference} is supported only on the
residue classes $0,1,3,4$ modulo $6$.  The second term is supported only on
$0$ and $2$ modulo $6$.  Therefore the coefficient of $q^{6n+5}$ in
$\mathcal D(\omega;q)$ is zero.  Since
$d_1^{(3)}(6n+5)=d_2^{(3)}(6n+5)$, we obtain
\[
 d_0^{(3)}(6n+5)=d_1^{(3)}(6n+5)=d_2^{(3)}(6n+5).
\]
Their sum is $\cthreeo(6n+5)$, which proves
\eqref{eq:cubic-equidistribution}.

For the equivalent divisibility statement, put
\[
 P_n(z)=\sum_{r\in\mathbb Z}d(r,6n+5)z^r.
\]
The preceding equality gives $P_n(\omega)=0$.  Since $P_n(z)$ has integer
coefficients and $1+z+z^2$ is the minimal polynomial of $\omega$ over
$\mathbb Q$, the Laurent polynomial $P_n(z)$ is divisible by $1+z+z^2$ in
$\mathbb Z[z,z^{-1}]$.
\end{proof}

\begin{proof}[Proof of Corollary~\ref{cor:cubic-sparse-rank-difference}]
Let $N\equiv2\pmod6$.  As in the preceding proof, the first term on the
right of \eqref{eq:D-omega-difference} is supported only on the residue
classes $0,1,3,4$ modulo $6$, and hence contributes nothing to the
coefficient of $q^N$.

It remains to examine $-\Phi_3(q^2)$.  By the definition of $\Phi_3$,
\[
 \Phi_3(q^2)=\sum_{m\ge0}(-1)^m
 \left(q^{2T_{3m}}+q^{2T_{3m+1}}\right).
\]
The exponents $2T_{3m}=3m(3m+1)$ are congruent to $0$ modulo $6$, whereas
\[
 2T_{3m+1}=(3m+1)(3m+2)\equiv2\pmod6.
\]
Therefore a term of degree $N\equiv2\pmod6$ occurs in
$\Phi_3(q^2)$ precisely when
\[
 N=(3m+1)(3m+2)
\]
for some $m\ge0$, and its coefficient is then $(-1)^m$.  The minus sign in
\eqref{eq:D-omega-difference} changes this coefficient to $(-1)^{m+1}$.
If no such $m$ exists, the coefficient is zero.  This proves
\eqref{eq:cubic-sparse-rank-difference}.
\end{proof}

The three-dimensional positive decomposition and sparse correction are
governed by the following analogous parity splitting.

\begin{lemma}
\label{lem:Cto-split}
Put
\[
Y_1(q)=(-q;q)_\infty^3
\sum_{m\ge0}\frac{\qpoch{q}{q^2}{m}q^m}{\qpoch{-q}{q}{m}^2},
\]
and
\[
Y_2(q)=\qinf{q^2}{q^2}\qinf{q}{q}
\sum_{m\ge0}\frac{\qpoch{q}{q^2}{m}q^m}{\qpoch{q}{q}{m}^2}.
\]
Then
\begin{equation}\label{eq:Y1-positive-split}
 Y_1(q)=E_{3,\triangle}(q)+2C_{3,o}(q),
\end{equation}
and
\begin{equation}\label{eq:Y2-signed-split}
 Y_2(q)=E_{3,\triangle}(q)-2C_{3,o}(q).
\end{equation}
In particular,
\[
 2C_{3,o}(q)=\frac12Y_1(q)-\frac12Y_2(q).
\]
\end{lemma}

\begin{proof}
Let $H_r(q)$ be the series in \eqref{eq:Hr-definition} and put
\[
 J_r(q):=(-1;q)_rH_r(q).
\]
Interchanging the sums in \eqref{eq:Hr-definition} and applying the
$q$-binomial theorem \eqref{eq:qbinomial} with $a=-1$ gives
\[
 \sum_{r\ge0}\frac{(-1;q)_r q^{r(m+1)}}{(q;q)_r}
 =\frac{(-q^{m+1};q)_\infty}{(q^{m+1};q)_\infty},
\]
whereas replacing $q^{m+1}$ by $-q^{m+1}$ gives
\[
 \sum_{r\ge0}\frac{(-1;q)_r(-q^{m+1})^r}{(q;q)_r}
 =\frac{(q^{m+1};q)_\infty}{(-q^{m+1};q)_\infty}.
\]
Using
\[
 (q^2;q^2)_m=(q;q)_m(-q;q)_m,
 \qquad
 (-q;q)_\infty=\frac1{(q;q^2)_\infty},
\]
the two resulting expressions simplify to
\begin{equation}\label{eq:Y1-Y2-Jr}
 Y_1(q)=\sum_{r\ge0}J_r(q),
 \qquad
 Y_2(q)=\sum_{r\ge0}(-1)^rJ_r(q).
\end{equation}

The formulas \eqref{eq:H-odd} and \eqref{eq:H-even} now identify the odd and
even parts.  For $r=2k+1$,
\begin{align*}
 J_{2k+1}(q)
 &=(-1;q)_{2k+1}
   \frac{q^{2k+1}}{(q;q^2)_{k+1}^2}\\
 &=2(-q;q)_{2k}
   \frac{q^{2k+1}}{(q;q^2)_{k+1}^2}.
\end{align*}
Thus
\begin{equation}\label{eq:J-odd-C3o}
 \sum_{k\ge0}J_{2k+1}(q)=2C_{3,o}(q).
\end{equation}
For $r=2k$,
\[
 J_{2k}(q)=(-1;q)_{2k}G(q)^2
 \frac{q^{2k}}{(q^2;q^2)_k^2},
\]
and therefore, by \eqref{eq:E3triangle-gf},
\begin{equation}\label{eq:J-even-E3triangle}
 \sum_{k\ge0}J_{2k}(q)=E_{3,\triangle}(q).
\end{equation}
Separating the even and odd indices in \eqref{eq:Y1-Y2-Jr}, and then using
\eqref{eq:J-odd-C3o} and \eqref{eq:J-even-E3triangle}, proves
\eqref{eq:Y1-positive-split} and \eqref{eq:Y2-signed-split}.
\end{proof}

\begin{lemma}\label{lem:Y2}
We have
\[
Y_2(q)=\sum_{j\ge0}(-1)^jq^{j(j+1)}.
\]
\end{lemma}

\begin{proof}
Again use the formal half-power convention from
Section~\ref{sec:rank-functions}.  Since
\[
\qpoch{q}{q^2}{n}=\qpoch{q^{1/2}}{q}{n}\qpoch{-q^{1/2}}{q}{n},
\]
apply Heine's transformation \eqref{eq:heine} with
\[
(a,b,c,z)=(-q^{1/2},q^{1/2},q,q).
\]
Then
\begin{align*}
Y_2(q)
&=\qinf{q^2}{q^2}\qinf{q}{q}
\sum_{n\ge0}\frac{\qpoch{q}{q^2}{n}q^n}{\qpoch{q}{q}{n}^2}  \\
&=\sum_{n\ge0}\frac{\qpoch{q^{1/2}}{q}{n}q^{n/2}}
{\qpoch{-q^{1/2}}{q}{n+1}}.
\end{align*}
The prefactor cancels for the same reason as in Lemma~\ref{lem:T2}.  Finally,
replacing $q$ by $q^{1/2}$ in \eqref{eq:ram-ch9-1} gives
\[
\sum_{n\ge0}\frac{\qpoch{q^{1/2}}{q}{n}q^{n/2}}
{\qpoch{-q^{1/2}}{q}{n+1}}
=\sum_{j\ge0}(-1)^jq^{j(j+1)}.
\]
\end{proof}

\begin{proof}[Proof of Theorem~\ref{thm:Pthdo-positive-decomposition}]
From \eqref{eq:Pthdo-gf},
\[
\sum_{n\ge0}\pthdo(n)q^n
=q(-q;q)_\infty^3
\sum_{m\ge0}\frac{\qpoch{q}{q^2}{m}q^m}{\qpoch{-q}{q}{m}^2}
=qY_1(q).
\]
Since $\pthdo(0)=0$, division by $q$ gives
\[
 \sum_{N\ge0}\pthdo(N+1)q^N=Y_1(q).
\]
Equation \eqref{eq:Y1-positive-split} now yields
\[
 \sum_{N\ge0}\pthdo(N+1)q^N
 =2C_{3,o}(q)+E_{3,\triangle}(q),
\]
which is \eqref{eq:Pthdo-positive-decomposition}.  Comparing coefficients
gives \eqref{eq:Pthdo-positive-coeff}.
\end{proof}

\begin{proof}[Proof of Theorem~\ref{thm:Cto-Pthdo}]
The preceding proof gives
\[
 \sum_{N\ge0}\pthdo(N+1)q^N=Y_1(q).
\]
By \eqref{eq:Y1-positive-split} and \eqref{eq:Y2-signed-split},
$Y_1(q)-4C_{3,o}(q)=Y_2(q)$.  Lemma~\ref{lem:Y2} therefore gives
\[
 \sum_{N\ge0}\left(\pthdo(N+1)-4\cthreeo(N)\right)q^N
 =\sum_{j\ge0}(-1)^jq^{j(j+1)}.
\]
The theorem follows by comparing coefficients.
\end{proof}

\begin{proof}[Proof of Corollary~\ref{cor:three-color-boxed}]
From the positive decomposition \eqref{eq:Pthdo-positive-coeff},
\[
 e_{3,\triangle}(N)=\pthdo(N+1)-2\cthreeo(N).
\]
Subtracting another $2\cthreeo(N)$ from both sides gives
\begin{equation}\label{eq:three-color-first-equality-proof}
 e_{3,\triangle}(N)-2\cthreeo(N)
 =\pthdo(N+1)-4\cthreeo(N).
\end{equation}

Next specialize Burson's identity \eqref{eq:burson-boxed} by taking
$(m,r)=(2,1)$ and $u=1$.  We obtain
\begin{align*}
&\sum_{M\ge0}
 \left(
 \sum_{(k,\overline\pi)\in\mathcal B_{2,1}(M)}
 (-1)^{\nu(\overline\pi)}
 \right)q^M\\
&\qquad=\sum_{j\ge0}(-1)^jq^{j(2j+2)}
 =\sum_{j\ge0}(-1)^jq^{2j(j+1)}.
\end{align*}
Taking the coefficient of $q^{2N}$ gives
\begin{align*}
 \sum_{(k,\overline\pi)\in\mathcal B_{2,1}(2N)}
 (-1)^{\nu(\overline\pi)}
 &= [q^{2N}]\sum_{j\ge0}(-1)^jq^{2j(j+1)}\\
 &= [q^N]\sum_{j\ge0}(-1)^jq^{j(j+1)}.
\end{align*}
By Theorem~\ref{thm:Cto-Pthdo}, the last coefficient is
$\pthdo(N+1)-4\cthreeo(N)$.  Together with
\eqref{eq:three-color-first-equality-proof}, this proves all equalities in
\eqref{eq:three-color-boxed}.

For completeness, Burson's involution pairs every nonfixed object with an
object of the same weight and opposite sign.  Its fixed point indexed by
$j$ is the pair $(j,(2j+1)^j)$.  In the present specialization its total
boxed weight is
\[
 j+j(2j+1)=2j(j+1),
\]
and its sign is $(-1)^j$.  Hence a fixed point contributes to the
coefficient of $q^{2N}$ exactly when $N=j(j+1)$, in agreement with the
sparse correction in Theorem~\ref{thm:Cto-Pthdo}.
\end{proof}

\begin{proof}[Proof of Corollary~\ref{cor:Pthdo-mod4}]
The case $N=0$ is immediate from the convention $\pthdo(0)=0$.  Now let
$N\ge1$.  Replacing $n$ by $N-1$ in
\eqref{eq:Pthdo-theta-coeff} gives
\begin{equation}\label{eq:Pthdo-mod4-proof-start}
 \pthdo(N)=4\cthreeo(N-1)+\varepsilon(N-1),
\end{equation}
where
\[
 \varepsilon(K)=
 \begin{cases}
 (-1)^j,&K=j(j+1)\text{ for some }j\ge0,\\
 0,&\text{otherwise.}
 \end{cases}
\]
Reducing \eqref{eq:Pthdo-mod4-proof-start} modulo $4$ eliminates the first
term and yields
\[
 \pthdo(N)\equiv\varepsilon(N-1)\pmod4.
\]

The condition $N-1=j(j+1)$ is equivalent to
\[
 4N-3=4j(j+1)+1=(2j+1)^2.
\]
Conversely, if $4N-3=x^2$ for a positive integer $x$, then $x$ is odd.
Writing $x=2j+1$ with $j\ge0$ gives
\[
 N-1=\frac{x^2-1}{4}=j(j+1).
\]
Thus the sparse correction is nonzero exactly when $4N-3$ is a positive
square.

In that case, with $x=2j+1$, one has $x\equiv1\pmod4$ when $j$ is even
and $x\equiv3\pmod4$ when $j$ is odd.  Therefore
\[
 \chi_4(x)=(-1)^j=\varepsilon(N-1),
\]
and hence $\pthdo(N)\equiv\chi_4(x)\pmod4$.  If $4N-3$ is not a square,
then $\varepsilon(N-1)=0$, so \eqref{eq:Pthdo-mod4-proof-start} gives
$\pthdo(N)\equiv0\pmod4$.  This proves both cases of the corollary.
\end{proof}

\begin{proof}[Proof of Corollary~\ref{cor:cubic-congruence-dissection}]
We first prove \eqref{eq:Cto-cubic-mod3}.  By the symmetry obtained by
interchanging the two side arms,
\[
 d_1^{(3)}(N)=d_2^{(3)}(N).
\]
Since the three rank classes exhaust all restricted odd three-dimensional
corners of weight $N$, it follows that
\begin{align*}
 \cthreeo(N)
 &=d_0^{(3)}(N)+d_1^{(3)}(N)+d_2^{(3)}(N)\\
 &=d_0^{(3)}(N)+2d_1^{(3)}(N)\\
 &\equiv d_0^{(3)}(N)-d_1^{(3)}(N)\pmod3.
\end{align*}
For $N\equiv2\pmod6$, Corollary~\ref{cor:cubic-sparse-rank-difference}
therefore gives
\[
 \cthreeo(N)\equiv
 \begin{cases}
 (-1)^{m+1},&N=(3m+1)(3m+2)\text{ for some }m\ge0,\\
 0,&\text{otherwise}
 \end{cases}
 \pmod3,
\]
which is \eqref{eq:Cto-cubic-mod3}.

We next prove \eqref{eq:Pthdo-cubic-mod12}.  Let $N\equiv3\pmod6$ and put
$K=N-1$, so that $K\equiv2\pmod6$.  Applying
\eqref{eq:Pthdo-theta-coeff} with $K$ in place of $n$ gives
\begin{equation}\label{eq:Pthdo-cubic-proof-start}
 \pthdo(N)=4\cthreeo(K)+\varepsilon(K),
\end{equation}
where $\varepsilon(K)=(-1)^j$ if $K=j(j+1)$ and is zero otherwise.

Suppose first that
\[
 K=(3m+1)(3m+2)
\]
for some $m\ge0$.  Then $K=j(j+1)$ with $j=3m+1$, and therefore
\[
 \varepsilon(K)=(-1)^{3m+1}=(-1)^{m+1}.
\]
The first part of the proof also gives
\[
 \cthreeo(K)\equiv(-1)^{m+1}\pmod3.
\]
Thus there is an integer $t$ such that
\[
 \cthreeo(K)=(-1)^{m+1}+3t.
\]
Substitution in \eqref{eq:Pthdo-cubic-proof-start} yields
\[
 \pthdo(N)=5(-1)^{m+1}+12t,
\]
and hence
\[
 \pthdo(N)\equiv5(-1)^{m+1}\pmod{12}.
\]

Now suppose that $K$ is not of the displayed form.  Because
$K\equiv2\pmod6$, any representation $K=j(j+1)$ would force
$j(j+1)\equiv2\pmod3$, and hence $j\equiv1\pmod3$.  Thus one would have
$j=3m+1$ and consequently $K=(3m+1)(3m+2)$, contrary to the assumption.
Therefore $\varepsilon(K)=0$.  The first part of the proof gives
$\cthreeo(K)\equiv0\pmod3$, so $4\cthreeo(K)$ is divisible by $12$.
Equation \eqref{eq:Pthdo-cubic-proof-start} now gives
$\pthdo(N)\equiv0\pmod{12}$.  This completes the proof.
\end{proof}

The proofs also explain the coefficients in the two discrepancy formulas.
The factor $2$ in Theorem~\ref{thm:Co-Ptdo} comes from the odd-index
projector $(1-(-1)^r)/2$ in Lemma~\ref{lem:Co-split}.  In the
three-dimensional case, the same projector is accompanied by the additional
factor $2$ in
\[
 (-1;q)_{2m+1}=2(-q;q)_{2m},
\]
which produces the coefficient $4$ in Theorem~\ref{thm:Cto-Pthdo}.  Thus
these coefficients arise from the parity-splitting transformations and are
not arbitrary normalizations.

\section{Proofs of the results in Section 7}
\label{sec:proofs-arithmetic}

We give the proofs in the order in which the principal results are stated.
The first auxiliary lemma shows that absence from the binary support also
excludes the unary correction support.

\begin{lemma}\label{lem:S-in-R}
For every $M\ge2$, one has $S_M\subseteq R_M$.
\end{lemma}

\begin{proof}
Fix $j$.  There is an odd nonnegative integer $t$ satisfying
$t\equiv j$ or $t\equiv-j-1\pmod M$.  If $M$ is odd, adding $M$ changes
parity; if $M$ is even, the two residue classes have opposite parity.  Write
$t=2a+1$ and take $b=t$.  Then
\[
T_{2a+1}+T_b=2T_t=t(t+1)\equiv j(j+1)\pmod M.
\]
Thus every residue represented by $j(j+1)$ is also represented by
$T_{2a+1}+T_b$.
\end{proof}

\begin{proof}[Proof of Theorem~\ref{thm:Cto-support-congruences}]
Set
\[
C_{3,o}(q)=\sum_{N\ge0}\cthreeo(N)q^N,
\qquad
P(q)=\qinf{q}{q^2},
\]
and
\[
A(q)=\sum_{j\ge0}(-1)^jq^{j(j+1)/2},
\qquad
B(q)=\sum_{j\ge0}(-1)^jq^{j(j+1)}.
\]
Putting $a=1$ in \eqref{eq:ram-part2-3} and using
\eqref{eq:Cto-gf} gives
\begin{equation}\label{eq:Cto-master}
2C_{3,o}(q)=\frac{A(q)}{P(q)^3}-B(q).
\end{equation}

We first prove the congruences modulo $3$.  Since
\[
P(q)^3=\qinf{q}{q^2}^3\equiv\qinf{q^3}{q^6}\pmod3,
\]
we have
\[
\frac{A(q)}{P(q)^3}
\equiv
\frac{A(q)}{\qinf{q^3}{q^6}}
\pmod3.
\]
The reciprocal of $\qinf{q^3}{q^6}$ contributes only powers divisible by
$3$, while
\[
\frac{j(j+1)}2\equiv0,1\pmod3.
\]
Hence $A(q)/P(q)^3$ has zero coefficient modulo $3$ in every degree
congruent to $2$ modulo $3$.  Suppose now that $3\mid M$,
$r\equiv2\pmod3$, and $r\notin S_M$.  Every integer $Mn+r$ is congruent to
$2$ modulo $3$, and $B(q)$ has no term in any degree congruent to $r$ modulo
$M$.  Taking the coefficient of $q^{Mn+r}$ in \eqref{eq:Cto-master} gives
\[
2\cthreeo(Mn+r)\equiv0\pmod3,
\]
and therefore
\[
\cthreeo(Mn+r)\equiv0\pmod3.
\]

We next prove the congruence modulo $2$.  Let
\[
U(q)=\sum_{j\ge0}q^{j(j+1)/2}
\]
and let
\[
V(q)=\sum_{j\ge0}q^{(2j+1)(2j+2)/2}
\]
be its odd-index subseries.  Then $A(q)=U(q)-2V(q)$.  By
\eqref{eq:gauss},
\[
U(q)=\frac{\qinf{q^2}{q^2}}{P(q)},
\]
so
\[
\frac{U(q)}{P(q)^3}=\frac{\qinf{q^2}{q^2}}{P(q)^4}.
\]
Modulo $4$,
\[
P(q)^4\equiv\qinf{q^2}{q^4}^2,
\]
and hence
\[
\frac{U(q)}{P(q)^3}
\equiv
\frac{\qinf{q^2}{q^2}}{\qinf{q^2}{q^4}^2}
=\frac{\qinf{q^4}{q^4}}{\qinf{q^2}{q^4}}
=U(q^2)
\pmod4.
\]
Moreover,
\[
\frac{1}{P(q)^3}\equiv U(q)\pmod2.
\]
Indeed, by \eqref{eq:gauss} this is equivalent to
\[
\frac{1}{\qinf{q}{q^2}^2}\equiv\qinf{q^2}{q^2}\pmod2,
\]
which follows from
\[
\qinf{q}{q^2}^2\equiv\qinf{q^2}{q^4}\pmod2
\]
and
\[
\frac{1}{\qinf{q^2}{q^4}}\equiv\qinf{q^2}{q^2}\pmod2.
\]
Consequently,
\begin{equation}\label{eq:Cto-mod4-A}
\frac{A(q)}{P(q)^3}
\equiv U(q^2)-2V(q)U(q)\pmod4.
\end{equation}
The support modulo $M$ of $U(q^2)$ is $S_M$, and that of $V(q)U(q)$ is
$R_M$.  If $r\notin R_M$, Lemma~\ref{lem:S-in-R} also gives $r\notin S_M$.
Thus neither term on the right of \eqref{eq:Cto-mod4-A}, nor the series
$B(q)$ in \eqref{eq:Cto-master}, contributes in a degree congruent to $r$
modulo $M$.  It follows that
\[
2\cthreeo(Mn+r)\equiv0\pmod4,
\]
and therefore
\[
\cthreeo(Mn+r)\equiv0\pmod2.
\]

It remains to transfer these conclusions to the three-color companion.  By
Theorem~\ref{thm:Cto-Pthdo},
\begin{equation}\label{eq:Pthdo-transfer-support}
\pthdo(Mn+r+1)=4\cthreeo(Mn+r)+\varepsilon(Mn+r),
\end{equation}
where $\varepsilon(K)$ is zero unless $K=j(j+1)$ for some $j\ge0$.
If $r\notin R_M$, then $r\notin S_M$ by Lemma~\ref{lem:S-in-R}; hence the
correction in \eqref{eq:Pthdo-transfer-support} vanishes and the parity
congruence for $\cthreeo(Mn+r)$ gives
\[
\pthdo(Mn+r+1)\equiv0\pmod8.
\]
If $3\mid M$, $r\equiv2\pmod3$, and $r\notin S_M$, the correction again
vanishes and the congruence modulo $3$ gives
\[
\pthdo(Mn+r+1)\equiv0\pmod{12}.
\]
When both hypotheses hold, the corresponding congruences combine to moduli
$6$ and $24$.  This proves the theorem.
\end{proof}

We now prove Theorem~\ref{thm:all-modulus-support} by reducing the
support problem to prime powers.  The next three lemmas give the local--global
reduction, the power-of-two case, and the required criterion at odd prime
powers.

\begin{lemma}\label{lem:R-local-global}
Let $M_1,\ldots,M_k$ be pairwise coprime positive integers and
$M=M_1\cdots M_k$.  A residue $r$ belongs to $R_M$ if and only if its
reduction belongs to $R_{M_i}$ for every $i$.
\end{lemma}

\begin{proof}
The forward implication is immediate.  Conversely, choose $a_i,b_i$ such that
\[
T_{2a_i+1}+T_{b_i}\equiv r\pmod{M_i}.
\]
Since
\[
T_{2a+1}=2a^2+3a+1
\]
is an integral polynomial in $a$, the ordinary Chinese remainder theorem
provides $a$ whose reduction modulo each $M_i$ is $a_i$.  The triangular
number $T_b$ has period $2M_i$ modulo $M_i$.  For every odd $M_i$, replacing
$b_i$ by $b_i+M_i$ preserves $T_{b_i}$ modulo $M_i$ and reverses the parity of
$b_i$.  We may therefore choose the representatives $b_i$ to have one common
parity.  The generalized Chinese remainder theorem then solves
\[
b\equiv b_i\pmod{2M_i}
\]
simultaneously, because the pairwise greatest common divisors of the moduli
$2M_i$ are $2$ and the residues have the same parity.  This $a,b$ represents
$r$ modulo every $M_i$, hence modulo $M$.
\end{proof}

\begin{lemma}\label{lem:R-two-power}
For every $\alpha\ge1$, the set $R_{2^\alpha}$ is the full residue system
modulo $2^\alpha$.
\end{lemma}

\begin{proof}
We first recall that triangular numbers are surjective modulo $2^\alpha$.
For a prescribed residue $r$, the integer $8r+1$ is congruent to $1$ modulo
$8$.  Every integer congruent to $1$ modulo $8$ is an odd square modulo every
higher power of $2$: starting with a solution modulo $8$, a solution $x$
modulo $2^k$ lifts to either $x$ or $x+2^{k-1}$ modulo $2^{k+1}$, since these
two squares differ by $2^k$ modulo $2^{k+1}$.  Hence there is an odd $x$ with
\[
x^2\equiv8r+1\pmod{2^{\alpha+3}}.
\]
Writing $x=2b+1$ and dividing by $8$ gives $T_b\equiv r\pmod{2^\alpha}$.
Now take $a=2^\alpha-1$.  Then
\[
T_{2a+1}=T_{2^{\alpha+1}-1}
=2^\alpha(2^{\alpha+1}-1)\equiv0\pmod{2^\alpha},
\]
so every residue belongs to $R_{2^\alpha}$.
\end{proof}

\begin{lemma}\label{lem:two-squares-prime-powers}
Let $p$ be an odd prime and $\alpha\ge1$.  If $p\equiv1\pmod4$, then
every residue modulo $p^\alpha$ is a sum of two squares.  If
$p\equiv3\pmod4$ and $t=v_p(A)<\alpha$, then $A$ is a sum of two squares
modulo $p^\alpha$ if and only if $t$ is even; residues divisible by
$p^\alpha$ are also sums of two squares.
\end{lemma}

\begin{proof}
Suppose first that $p\equiv1\pmod4$.  Then $-1$ has a square root $i$ modulo
$p$, and Hensel lifting gives $i^2\equiv-1\pmod {p^\alpha}$.  For any residue
$A$, put
\[
x=\frac{A+1}{2},\qquad y=\frac{A-1}{2i}\pmod {p^\alpha}.
\]
Then $x^2+y^2\equiv A\pmod {p^\alpha}$.

Now suppose $p\equiv3\pmod4$.  Every unit $U$ modulo $p^\beta$ is a sum of two
squares.  Modulo $p$, the two sets $\{x^2\}$ and $\{U-y^2\}$ each have
$(p+1)/2$ elements and hence intersect.  The resulting solution has at least
one nonzero partial derivative, so it lifts to every $p^\beta$.

If $A=p^{2k}U$ with $U$ a unit and $2k<\alpha$, represent $U$ modulo
$p^{\alpha-2k}$ as $X^2+Y^2$ and multiply by $p^{2k}$.  This proves
representability when $v_p(A)$ is even, while $A\equiv0\pmod {p^\alpha}$ is
represented by $0^2+0^2$.

Conversely, suppose $x^2+y^2\equiv A\pmod {p^\alpha}$ and
$t=v_p(A)<\alpha$.  Let
$k=\min(v_p(x),v_p(y))$.  After dividing by $p^k$, at least one of the two
new variables is nonzero modulo $p$, and their sum of squares is nonzero
modulo $p$, since $-1$ is a nonsquare modulo $p$.  Hence
$v_p(x^2+y^2)=2k$.  Since $x^2+y^2\equiv A\pmod {p^\alpha}$ and
$t<\alpha$, we have
$t=2k$, which is even.
\end{proof}

\begin{proof}[Proof of Theorem~\ref{thm:all-modulus-support}]
By Lemma~\ref{lem:R-local-global}, it is enough to classify the image modulo
prime powers.  Lemma~\ref{lem:R-two-power} shows that the power of $2$ creates
no obstruction.

Let $p$ be odd.  Since $2$ and $4$ are invertible modulo $p^\alpha$, one has
\[
r\in R_{p^\alpha}
\quad\Longleftrightarrow\quad
8r+2\text{ is a sum of two squares modulo }p^\alpha.
\]
Indeed,
\[
8T_{2u+1}+1=(4u+3)^2,
\qquad
8T_v+1=(2v+1)^2,
\]
and the residues $4u+3$ and $2v+1$ run through all classes modulo
$p^\alpha$.

If $p\equiv1\pmod4$, Lemma~\ref{lem:two-squares-prime-powers} shows that
$R_{p^\alpha}$ is full.  If $p\equiv3\pmod4$, the same lemma shows that
$8r+2=2(4r+1)$ fails to be a sum of two squares exactly when
$v_p(4r+1)$ is odd and less than $\alpha$.  Combining the local conditions by
Lemma~\ref{lem:R-local-global} proves the equivalence.

If every prime $p\equiv3\pmod4$ has exponent at most $1$, no local
obstruction is possible.  Conversely, if $p^\alpha\parallel M$ with
$p\equiv3\pmod4$ and $\alpha\ge2$, choose $r$ modulo $p^\alpha$ so that
$4r+1\equiv p\pmod{p^\alpha}$ and choose its residues arbitrarily at the other
prime-power factors.  The Chinese remainder theorem then gives a class modulo
$M$ with $v_p(4r+1)=1<\alpha$, so $R_M$ is not full.
\end{proof}

\begin{proof}[Proof of Corollary~\ref{cor:Cto-prime-powers}]
The valuation hypothesis and Theorem~\ref{thm:all-modulus-support} give
$r\notin R_{p^\alpha}$.  Part~(1) of
Theorem~\ref{thm:Cto-support-congruences} therefore gives the congruences
modulo $2$ and $8$.

If $p=3$, the divisibility $3\mid4r+1$ implies $r\equiv2\pmod3$.  Moreover,
$r\notin R_{3^\alpha}$ and Lemma~\ref{lem:S-in-R} imply
$r\notin S_{3^\alpha}$.
Part~(2) of Theorem~\ref{thm:Cto-support-congruences} then gives the
congruences modulo $3$ and $12$.  Combining the two pairs yields the stated
moduli $6$ and $24$.
\end{proof}

\clearpage

\section{Concluding remarks}

The corner viewpoint gives a parallel world for a substantial part of the
third-order mock-theta family.  The parallel is exact at the level of the
finite rank kernels: the quadratic core of the mock-theta series is replaced
by a linear corner peak, and Ramanujan's Lost Notebook supplies the resulting
partial- and false-theta evaluations.  The exceptional identity
$C_\psi(q)=q\omega_3(q)$ shows that the passage does not invariably leave the
mock-theta world.

The positive decompositions in
Theorems~\ref{thm:Ptdo-positive-decomposition} and
\ref{thm:Pthdo-positive-decomposition} rewrite the two colored companions as
\[
 C_\omega(q)+E_\triangle(q)
 \qquad\text{and}\qquad
 2C_{3,o}(q)+E_{3,\triangle}(q),
\]
respectively.  Consequently, both sparse corrections are signed differences
between natural corner families.  The correction terms in
Theorems~\ref{thm:Co-Ptdo} and \ref{thm:Cto-Pthdo} are also governed by one
boxed modular mechanism.
Burson's involutions cancel all nonfixed overpartition pairs, and the fixed
pairs \eqref{eq:burson-fixed-pairs} account for the two quadratic families in
the two-color identity and the single quadratic family in the three-color
identity.  Corollaries~\ref{cor:two-color-boxed} and
\ref{cor:three-color-boxed} therefore give auxiliary involutive models for
the sparse errors.  The remaining direct problems are now parallel: one seeks an involution
between staircase-decorated even corners and odd corners in the two-color
case, and an involution between staircase-and-tower-decorated even corners
and two labeled copies of the restricted odd three-dimensional corners in
the three-color case.  Such constructions would lift the boxed fixed-point
models back to the original corner objects and explain the factors $2$ and
$4$ without passing through generating functions.

The exact specialization at $z=-1$ and the quartic and cubic identities
in Theorems~\ref{thm:D-parity-specialization},
\ref{thm:D-parity-bias}, and~\ref{thm:D-cyclotomic-differences} show that the
two-variable three-dimensional series contains exact information not visible
from its specialization at $z=1$.  In particular, the cubic specialization
produces both exact rank equidistribution and a sparse rank difference.
Several questions remain.  Can the two corner-to-corner differences in
Corollaries~\ref{cor:two-color-boxed} and~\ref{cor:three-color-boxed} be
proved by direct sign-reversing involutions, and can the resulting positive
decompositions be realized by explicit bijections from the original colored
companions?  Can kernel-preserving linearization be realized by a
common Bailey-pair, Hecke--Appell, or Appell--Lerch mechanism for the complete
identities rather than only their correction terms?  What further rank
dissections follow from higher roots of unity, and which congruences arise
from cancellation inside $R_M$ rather than merely from support exclusion?
Finally, the corresponding one-arm linearization leads naturally to
fifth-order mock theta functions.  That direction appears sufficiently rich
to merit a separate treatment.


\begin{thebibliography}{99}

\bibitem{Dyson}
F.~J. Dyson,
Some guesses in the theory of partitions,
\emph{Eureka} \textbf{8} (1944), 10--15.

\bibitem{partitions}
G.~E. Andrews,
\emph{The Theory of Partitions}, Encyclopedia of Mathematics and its
Applications, Vol.~2, Addison-Wesley, Reading, MA, 1976; reprinted by
Cambridge University Press, Cambridge, 1998.

\bibitem{A84}
G.~E. Andrews,
Ramanujan's ``lost'' notebook. IV. Stacks and alternating parity in partitions,
\emph{Adv. Math.} \textbf{53} (1984), no.~1, 55--74.


\bibitem{AndrewsCC}
G.~E. Andrews,
Concave and convex compositions,
\emph{Ramanujan J.} \textbf{31} (2013), no.~1--2, 67--82.

\bibitem{AndrewsFour}
G.~E. Andrews, B.~C. Berndt, S.~H. Chan, S.~Kim, and A.~Malik,
Four identities for third-order mock theta functions,
\emph{Nagoya Math. J.} \textbf{239} (2020), 173--204.

\bibitem{AllenOsburn}
K.~Allen and R.~Osburn,
Unimodal sequences and mixed false theta functions,
\emph{Adv. Math.} \textbf{473} (2025), Article 110293.

\bibitem{AB1}
G.~E. Andrews and B.~C. Berndt,
\emph{Ramanujan's Lost Notebook, Part I}, Springer, New York, 2005.

\bibitem{AB2}
G.~E. Andrews and B.~C. Berndt,
\emph{Ramanujan's Lost Notebook, Part II}, Springer, New York, 2009.

\bibitem{AB3}
G.~E. Andrews and B.~C. Berndt,
\emph{Ramanujan's Lost Notebook, Part III}, Springer, New York, 2012.


\bibitem{AN}
G.~E. Andrews and D.~Newman,
The minimal excludant in integer partitions,
\emph{J. Integer Seq.} \textbf{23} (2020), Article 20.2.3.

\bibitem{Au}
F.~C. Auluck,
On some new types of partitions associated with generalized Ferrers graphs,
\emph{Proc. Cambridge Philos. Soc.} \textbf{47} (1951), 679--686.

\bibitem{B}
B.~C. Berndt,
\emph{Ramanujan's Notebooks, Part IV}, Springer, New York, 1994.


\bibitem{BridgesBringmann}
W.~Bridges and K.~Bringmann,
Taylor coefficients of false Jacobi forms and ranks of unimodal sequences,
\emph{Adv. Math.} \textbf{426} (2023), Article 109074.

\bibitem{Burson}
H.~E. Burson,
A bijective proof of a false theta function identity from Ramanujan's lost
notebook,
\emph{Ann. Comb.} \textbf{23} (2019), no.~3--4, 579--588.

\bibitem{CorteelLovejoy}
S.~Corteel and J.~Lovejoy,
Overpartitions,
\emph{Trans. Amer. Math. Soc.} \textbf{356} (2004), no.~4, 1623--1635.

\bibitem{BLodd}
K.~Bringmann and J.~Lovejoy,
Odd unimodal sequences,
\emph{Adv. Math.} \textbf{480}, Part B (2025), Article 110458.

\bibitem{BPV}
K.~Bringmann, B.~V. Pandey, and J.-W. van Ittersum,
False and partial Eisenstein-type series related to unimodal sequences,
\emph{Res. Math. Sci.} \textbf{13} (2026), Article 46.


\bibitem{DGLP}
S.~Datta, M.~R. Gaberdiel, W.~Li, and C.~Peng,
Twisted sectors from plane partitions,
\emph{J. High Energy Phys.} 2016, no.~9, Paper No.~138.

\bibitem{KLrank}
B.~Kim and J.~Lovejoy,
The rank of a unimodal sequence and a partial theta identity of Ramanujan,
\emph{Int. J. Number Theory} \textbf{10} (2014), no.~4, 1081--1098.

\bibitem{KLspecial}
B.~Kim and J.~Lovejoy,
Ramanujan-type partial theta identities and rank differences for special
unimodal sequences,
\emph{Ann. Comb.} \textbf{19} (2015), no.~4, 705--733.


\bibitem{K}
G.~Kreweras,
Sur les extensions lin\'eaires d'une famille particuli\`ere d'ordres partiels,
\emph{Discrete Math.} \textbf{27} (1979), 279--295.

\bibitem{O}
OEIS Foundation Inc.,
The On-Line Encyclopedia of Integer Sequences, A006330.

\bibitem{S}
R.~P. Stanley,
\emph{Enumerative Combinatorics, Vol.~1}, Cambridge University Press,
Cambridge, 1999.

\bibitem{T}
H.~N.~V. Temperley,
Statistical mechanics and the partition of numbers II. The form of crystal surfaces,
\emph{Proc. Cambridge Philos. Soc.} \textbf{48} (1952), 683--697.

\bibitem{FMR}
A.~Folsom, J.~Males, and L.~Rolen,
Equidistribution and partition polynomials,
\emph{Ramanujan J.} \textbf{65} (2024), 1827--1848.

\bibitem{GarvanRank}
F.~G. Garvan,
Transformation properties for Dyson's rank function,
\emph{Trans. Amer. Math. Soc.} \textbf{371} (2019), 199--248.

\bibitem{Watson}
G.~N. Watson,
The final problem: an account of the mock theta functions,
\emph{J. London Math. Soc.} \textbf{11} (1936), 55--80.

\bibitem{WrightStacks}
E.~M. Wright,
Stacks,
\emph{Quart. J. Math. Oxford Ser. (2)} \textbf{19} (1968), 313--320.

\enlargethispage{6\baselineskip}
\end{thebibliography}
\end{document}